\documentclass[12pt]{amsart}
\usepackage{amsmath,amsthm,amscd,amssymb,bbm,graphicx}
\usepackage{floatflt,setspace, wrapfig}
\usepackage{graphics,multicol}%changepage}
\usepackage{mathrsfs} %gives fance script A font for attractor
\usepackage{epstopdf}
\usepackage{geometry}

\usepackage{todonotes}

\newtheorem{thrm}{Theorem}[section]

\newtheorem{prop}{Proposition}[section]
\newtheorem{lem}{Lemma}[section]

\newcommand{\ind}{\mathbbm{1}}

\newcommand{\ball}{B_{\rho}}
\newcommand{\diam}{\mbox{diam }}
\newcommand{\var}{\mbox{var}}
\newcommand{\dist}{\mbox{dist}}

\providecommand{\abs}[1]{\lvert \, #1 \, \rvert}

\providecommand{\norm}[1]{\lVert \, #1 \, \rVert}
\providecommand{\floor}[1]{\left\lfloor \, #1 \, \right\rfloor}
\providecommand{\roof}[1]{\left\lceil \, #1 \, \right\rceil}

\newcommand{\bea}[1]{\begin{eqnarray}\label{#1}}
\newcommand{\eea}{\end{eqnarray}}

\title[Exponential Law for Random Maps]{Exponential Law for Random  Maps  on Compact Manifolds}
 \date{\today}
 \thanks{This work was partially supported by FAPESB, CNPq, CAPES, FCT project PTDC/MAT-PUR/28177/2017, with national funds, and by CMUP (UID/MAT/00144/2019), which is funded by FCT with national (MCTES) and European structural funds through the programs FEDER, under the partnership agreement PT2020.}
 
\begin{document}
 \maketitle
\authors{Nicolai T A Haydn\footnote{Department of Mathematics, University of Southern California,
Los Angeles, 90089-2532. E-mail: {\tt \email{nhaydn@usc.edu}}.},
J\'er\^ome Rousseau\footnote{Departamento de Matematica, Universidade Federal da Bahia, Avenida Ademar de Barros s/n, 40170-110 Salvador, BA, Brasil. E-mail: {\tt \email{jerome.rousseau@ufba.br}}}\footnote{Departamento de Matem\'atica, Faculdade de Ci\^encias da Universidade do Porto,Rua do Campo Alegre, 687, 4169-007 Porto, Portugal},
Fan Yang\footnote{Department of Mathematics, University of Oklahoma,
Norman, OK 73019-3103. E-mail: {\tt \email{fan.yang-2@ou.edu}}.}}

%\tableofcontents

%%%%%%%%%%%%%%%%%%%%%%%%%%%%%%%%%%%%%%%%%%%%
%%%%%%%%%%%%%%%%%%%%%%%%%%%%%%%%%%%%%%%%%%%%
%%%%%%%%%%%%%%%%%%        ABSTRACT
\begin{abstract}
We consider random dynamical systems on manifolds modeled by a skew product which have certain geometric properties and whose measures satisfy quenched decay of correlations at a sufficient rate. We prove that the limiting distribution for the hitting and return times to geometric balls are both exponential for almost every realisation. We then apply this result to random $C^2$ maps of the interval, random parabolic maps on the unit interval, random perturbation of partially hyperbolic attractors on a compact Riemannian manifold and random perturbation of non-uniformly expanding maps with critical set.
\end{abstract}

\tableofcontents

%%%%%%%%%%%%%%%%%%%%%%%%%%%%%%%%%%%%%%%%%%%%
%%%%%%%%%%%%%%%%%%%%%%%%%%%%%%%%%%%%%%%%%%%%
%%%%%%%%%%%%%%%%%%             INTRODUCTION
\section{Introduction}
As a generalisation of deterministic dynamical systems but also as a better approximation of natural phenomena (e.g. existence of smalls perturbations), random dynamical systems have been extensively study in the last few decades. Unlike deterministic systems which only consider the iteration of one map, random systems allow the composition of different maps (for example by adding the presence of random noise or random perturbations) thus increasing the difficulty to analyze their statistical properties, in particular since these maps generally do not share a common invariant measure. We point the readers to the review paper by Kifer and Liu~\cite{KL06} for more details.

Among these statistical properties, we would like to focus on limit laws for rare events, more precisely Hitting Time Statistics (HTS) and Return Time Statistics (RTS). The study of hitting/return times for deterministic systems traces all the way back to the famous work of Poincar\'e~\cite{P99} who proved that the orbit of almost every point comes back as close as you want to its starting point and our main subject of interest will be the time needed for the orbit to come back. More precisely, if we denote by $\tau_A(x)$ the first time that the orbit of $x$ enters the set $A$,
then one can consider the functions
$$
F^h_A(t) = \mu\!\left(x\in X: \tau_{A}(x)>\frac{t}{\mu(A)}\right)
$$
and
$$
F^r_A(t) = \frac{1}{\mu(A)}\mu\!\left(x\in A: \tau_{A}(x)>\frac{t}{\mu(A)}\right)
$$
where $\mu$ is an invariant measure of the transformation. We can observe that the scaling factor $\frac{1}{\mu(A)}$ is suggested by the Kac's Lemma~\cite{K47} which says that $\int_A\tau_A\,d\mu=1$.

One is naturally interested whether the functions $F^h_{A_n}$ (respectively $F^r_{A_n}$) converge to a limiting function $F^h$ (respectively $F^r$)
as $n\to\infty$ when one chooses a sequence of nested sets $\{A_n\}_{n=1}^\infty$.
Indeed, if the sets $A_n$ are taken to be cylinder 
sets with respect to a generating measurable partition, then the limit is known to be exponential 
for non-periodic points for mixing measures (see e.g.~\cite{GS97,Ab2}). In the case of Bowen 
balls the same result is known to be true~\cite{HY}. For geometric balls $B_r(y)$ it has been proven that the limit $F$ is exponential if $\mu$ has exponential decay of correlations (e.g. \cite{Rou14} and references therein). We refer to the reviews \cite{H13, Sau09} for more details on this subject. 

Our goal in this paper is to extend some of these results for deterministic dynamical systems to the realm of random dynamical systems. 
We consider a family of maps $\{T_\omega\}_\omega$ with $T_\omega:X\rightarrow X$. The randomness comes from a dynamical system $(\Omega, \theta, \nu)$ and the random orbit is given by $$T_\omega^n(x)=T_{\theta^{n}\omega}\circ T_{\theta^{n-1}\omega}\circ\dots\circ T_\omega(x).$$
Thus, one can define the (quenched) hitting time $\tau^\omega_A(x)$ as the first time the random orbit of $x$ enters the 
set $A$. 
There are two ways to define the hitting times distribution, namely 
$$
F(t) = \mathbb{P}\!\left((\omega,x): \tau^\omega_{A}(x)>\frac{t}{\mu(A)}\right)
$$
and
$$
F^\omega(t) = \mu^\omega\!\left(x: \tau^\omega_{A}(x)>\frac{t}{\mu(A)}\right).
$$
The first is known as the annealed distribution, where the probability is taken with respect to the measure $\mathbb{P}$, which is invariant for the random dynamical system, i.e. invariant for the associated skew-product. The second is called the quenched distribution where the probability is taken
with the measure $\mu^\omega$ associated with the `realisation' $\omega$. In both cases, the scaling factor is $\frac{1}{\mu(A)}$, where $\mu$ is the marginal measure and is suggested by Kac's Lemma for the associated skew-product \cite{MarieR,RSV}. 

In \cite{AFV,AA, Rou14}, it is proven that the annealed distribution for geometric balls 
converges to an exponential for maps with (annealed) exponential decay of correlations. 
In the first two papers, their method exploits the relation between the hitting times statistics and the extreme 
value distribution, while in the third one the method of \cite{HSV} is followed. In these paper, the convergence of the return time distribution to an exponential is also proven.

On the other hand, a quenched result is more interesting since it easily implies the annealed 
result by integrating over $\omega$, but more difficult to get. The only known results are~\cite{RSV,RT,FFV17} 
where random subshifts of finite type with fast decay of correlations are considered and an exponential law is proved for hitting times. 

We emphasize that in these articles studying the quenched case, they did not prove the distribution for the return times and more importantly one can notice that the convergence of the return time distribution does not come immediately from the convergence of the hitting time distribution, as one could have hoped for from the deterministic case (e.g. \cite{HLV,Sau09}) or the annealed case \cite{Rou14}.

In this paper, we extend the results of \cite{RSV} to maps and prove that the quenched hitting time statistics converges, for almost every $\omega$, to the exponential distribution for random maps which have certain geometric properties and with some rapidly mixing conditions. Moreover, this is the first paper where one also managed to obtain the convergence of the quenched return time statistics to the exponential distribution. The main theorems are stated in Section~\ref{random.maps},
and proven in Sections~\ref{very.short.returns} to \ref{return.times.distribution}. The proof is based on the deterministic case~\cite{HY17} which
in its turn was derived from the deterministic case on Young towers~\cite{HW14} but we emphasize that this is not just a simple adaptation of the deterministic case to the random case. In particular, one needs to be especially careful in the study of the very short returns (Section \ref{very.short.returns}) and we encourage the reader to take a particular attention to Proposition \ref{prop.short.returns} which is a complex adaptation of a lemma of \cite{CC13} on the measure of the set of very short returns.
In Section~\ref{example} we consider four examples, namely  random $C^2$ expanding interval maps and  random Pomeau-Manneville maps where we use the derivation of the fibered measures
from~\cite{BB16} and the decay of sequential systems for parabolic maps~\cite{AHNTV14}; the third example is the random perturbation of non-uniformly expanding maps with critical set, which can be modeled by a random Gibbs-Markov-Young structure~\cite{LV}; the final example is the random perturbation of partially hyperbolic attractors, where the system exhibits non-trivial stable and unstable leaves depending on $\omega$.

%%%%%%%%%%%%%%%%%%%%%%%%%%%%%%%%%%%%%%%%%%%%%%%%%%%%%%%%%%%%%%%%%%%%%%%
%%%%%%%%%%%%%%%%%%%%%%%%%%%%%%%%%%%%%%%%%%%%%%%%%%%%%%%%%%%%%%%%%%%%%%% Background
\section{Random Maps} \label{random.maps}

Let $\theta:\Omega\to\Omega$ be the shift map on a full shift space $\Omega$ 
with $\theta$-invariant probability measure $\nu$. Let $M$ be a compact manifold
and for every $\omega\in\Omega$, let $T_\omega:M\to M$ be a measurable map. 
The skew product $S$ on $\Omega\times M$ is then given by 
$S(\omega,x)=(\theta\omega,T_\omega x)$.
For the iterates we obtain  $S^n(\omega,x)=(\theta^n\omega,T_\omega^nx)$ where
$T_\omega^n=T_{\theta^{n-1}\omega}\circ\cdots\circ T_{\theta\omega}\circ T_\omega$.

Assume that $\mathbb{P}$ is a measure on $\Omega\times M$ 
invariant under the skew action $S$ and with marginal $\nu$ on $\Omega$. There is a class of measures $\mu^\omega$ for $\omega\in \Omega$ on
$M$, such that $d\mathbb{P}=d\mu^\omega d\nu(\omega)$.
These measures satisfy the invariance property $T_\omega^*\mu^\omega=\mu^{\theta\omega}$ for $\nu$-almost every $\omega\in \Omega$ .
We denote by $\mu=\int_\Omega \mu^\omega\,d\nu(\omega)$ the marginal measure on $M$.

For every realisation $\omega\in\Omega$ let $\Gamma^u(\omega)$ be a collection of 
unstable leaves $\gamma^u(\omega)$
and  $\Gamma^s(\omega)$ a collection of stable leaves $\gamma^s(\omega)$. We assume
that $\gamma^u\cap\gamma^s$ consists of a single point for all $(\gamma^u,\gamma^s)\in\Gamma^u\times\Gamma^s$. 
The map $T_\omega$ contracts along the stable leaves and similarly $T_\omega^{-1}$ 
contracts along the unstable leaves.

For an unstable leaf $\gamma^u(\omega)$ denote by $\mu^\omega_{\gamma^u}$ the 
disintegration of $\mu^\omega$ with respect to the $\gamma^u$. We assume that $\mu^\omega$ 
has a product like decomposition  
$d\mu^\omega=d\mu^\omega_{\gamma^u}d\upsilon^\omega(\gamma^u)$,
where $\upsilon^\omega$ is a transversal measure. That is, if $f$ is a function on $M$ then
$$
\int f(x)\,d\mu^\omega(x)
=\int_{\Gamma^u(\omega)} \int_{\gamma^u}f(x)\,d\mu^\omega_{\gamma^u}(x)\,d\upsilon^\omega(\gamma^u)
$$

If $\gamma^u, \hat\gamma^u\in\Gamma^u(\omega)$ are two unstable leaves then the 
holonomy map $\mathcal{H}:\gamma^u\cap\Lambda\to \hat\gamma^u\cap\Lambda$ 
is defined by $\mathcal{H}(x)=\hat\gamma^u\cap\gamma^s(x)$ for $x\in\gamma^u\cap\Lambda$,  
where $\gamma^s(x)$ is the local stable leaf through $x$.

Let us denote by 
$J_n^\omega=\frac{dT_\omega^n\mu^\omega_{\gamma^u}}{d\mu^\omega_{\gamma^u}}$
the Jacobian of the map $T_\omega^n$ with respect to the measure $\mu^\omega$
in the unstable direction.

Fix $\omega$ and let $\gamma^u$ be a local unstable leaf.
Assume there exists $R>0$ and  for every $n\in\mathbb{N}$ finitely many
 $y_k\in T_\omega^n\gamma^u$ so that 
$T_\omega^n\gamma^u\subset\bigcup_k B_{R,\gamma^u}(y_k)$, 
where $B_{R,\gamma^u}(y)$ is the embedded $R$-disk centered at $y$ in the unstable leaf
 $\gamma^u$.
Denote by $\zeta_{\varphi,k}=\varphi(B_{R,\gamma^u}(y_k))$ where $\varphi\in \mathscr{I}_n^\omega$
 and $\mathscr{I}_n^\omega$ denotes the  inverse branches of $T_\omega^n$. 
 We call $\zeta$ an $n$-cylinder.
Then there exists a constant $L$ so that the number of overlaps
$N_{\varphi,k}=|\{\zeta_{\varphi',k'}: \zeta_{\varphi,k}\cap\zeta_{\varphi',k'}\not=\varnothing,
\varphi'\in\mathscr{I}^\omega_n\}|$
is bounded by $L$ for all $\varphi\in \mathscr{I}_n^\omega$ and for all $k$ and $n$. 
This follows from the fact that
$N_{\varphi,k}$ equals $|\{k': B_{R,\gamma^u}(y_k)\cap B_{R,\gamma^u}(y_{k'})\not=\varnothing\}|$ 
which is uniformly bounded by some constant $L$.

%%%%%%%%%%%%%%%%%%%%%%%%%%%%%%%%%%%%%%%%%%%%%%%
%%%%%%%%%%%%%%%%%%%%%%%%%%%%%%%%%%%%%%%%%%%%%%%
%%%%%%%%%%%%              ASSUMPTIONS
To obtain an exponential law for the distribution of hitting time and return time, we need a few assumptions. First of all, we need information on the annealed and quenched decay of correlations:\\
(I) There exists a decay function $\lambda(k)$ so that 
$$
\left|\int_\Omega\int_MG(H\circ T_\omega^k)\,d\mu^\omega\,d\nu(\omega)
-\mu(G)\mu(H)\right|
\le \lambda(k)\|G\|_{Lip}\|H\|_\infty\qquad\forall k\in\mathbb{N}
$$
for every $G\in Lip(M,\mathbb{R})$ and $H\in L^\infty(M,\mathbb{R})$.\\
(II) For $\nu$-almost every $\omega$, the individual measure $\mu^\omega$ has the following decay of correlations
$$
\left|\int_MG(H\circ T_\omega^k)\,d\mu^\omega
-\mu^\omega(G)\mu^{\theta^k\omega}(H)\right|
\le \lambda(k)\|G\|_{Lip}\|H\|_\infty\qquad\forall k\in\mathbb{N},
$$
for every $H\in L^\infty(M,\mathbb{R})$ which are constant on 
local stable leaves $\gamma^s$ of $T_\omega$ and for every $G\in Lip(M,\mathbb{R})$.\\
Then, we need some geometric assumptions:\\
(III)  (Distortion) For $\nu$-almost every $\omega$, we  require that 
$\frac{J^\omega_n(x)}{J^\omega_n(y)}=\mathcal{O}(\Theta(n))$ for all $x,y\in\zeta$ and $n$, where 
$\zeta$ are $n$-cylinders in unstable leaves $\gamma^u$ and $\Theta$ is a non-decreasing 
function which below we assume to be $\Theta(n)=\mathcal{O}(n^{\kappa'})$ for 
some $\kappa'\ge0$. \\
(IV) (Contraction) There exists a function $\delta(n)\to0$ which decays at least summably polynomially, i.\,e.\, $\delta(n) = \mathcal{O}(n^{-\kappa})$ with $\kappa > 1$, so that 
$\diam\zeta\le \delta(n)$ for all $n$-cylinder $\zeta$ and all $n$ and $\omega$.\\
Finally, we need some informations on the measures:\\
(V) There exist $0<d_0<d_1$ and $K$ such that $\rho^{d_0}\ge\mu(\ball)\ge\rho^{d_1}$ and 
$$
\frac1K\le\frac{\mu(B_{\rho})}{\mu^\omega(\ball)}\leq K
$$ 
for all $\rho>0$ small enough and for $\nu$-almost every $\omega$.\\
%(IV) There exists a $u_0$ so that 
%$\mu^\omega_{\gamma^u}(B_\rho(x))\le C_1\rho^{u_0}$ for all $\rho>0$ small enough
%and for almost all $x\in \gamma^u$, every unstable leaf $\gamma^u$ and $\nu$-almost all $\omega$.\\
(VI) (Annulus condition) Assume that for some $\xi\ge\beta>0$:
$$
\sup_\omega\frac{\mu^\omega(B_{\rho+r}\setminus B_{\rho-r})}{\mu(\ball)} 
= \mathcal{O}(\frac{r^\xi}{\rho^\beta})
$$
for every $r < \rho$.\\
%(VIII) \\
%To get the distribution of the return times, we will need one additional assumption:\\
%(IX) (Random annulus condition) Assume that for some $\xi\ge \beta>0$:
%$$
%\frac{\mu^\omega(B_{\rho+r}\setminus B_{\rho-r})}{\mu^\omega(\ball)} = \mathcal{O}(\frac{r^\xi}{\rho^\beta})
%$$
%for every $r < \rho$ and $\nu$-almost all $\omega$.

Our main result is on the distribution of the first hitting and return times.
For a set $B\subset M$  $\omega\in\Omega$ one defines the function
$$
\tau^\omega_B(x)=\inf\{j\ge1:\; T_\omega^jx\in B\}.
$$
This is the {\em hitting time function} on $M$ or the {\em  return time function}
when restricted to $B$ itself. We can now state our main results, here $\mu^\omega_B$ is 
the conditional measure of $\mu^\omega$ restricted to the set $B\subset M$. First of all, under the previous assumption we obtain an exponential law for the distribution of hitting times.

%%%%%%%%%%%%%%%%%%%%%%%%%%%%%%%%%%%%%%%%%%%%%
%%%%%%%%%%%%%%%%%%%%%%%%%%%%%%%%%%%%%%%%%%%%%
%%%%%%%%%%  THEOREM: MAIN THEOREM
%%%%%%%%%%%%%%%%%%%%%%%%%%%%%%%%%%%%%%%%%%%%%

\begin{thrm}\label{main.theorem}
Let a random dynamical system satisfy the above requirements~(I)--(VI) where $\delta$ and $\lambda$ both decay super polynomially fast.
%Let $\rho_j\to 0$ be so that $\sum_j\rho_j^q<\infty$ for some positive 
%$q<d_0$.

Then
$$
\mu^\omega\!\left(y\in M: 
\tau^\omega_{B_{\rho}(\mathsf{x})}(y)>\frac{t}{\mu(B_{\rho}(\mathsf{x}))}\right)
\longrightarrow e^{-t}\qquad \mbox{ as }\rho\to0
$$
%and

%$$
%\mu^\omega_{B_{\rho}(\mathsf{x})}\!\left(y\in M: 
%\tau^\omega_{B_{\rho}(\mathsf{x})}(y)>\frac{t}{\mu(B_{\rho}(\mathsf{x}))}\right)
%\longrightarrow e^{-t}\qquad \mbox{ as }\rho\to0
%$$
for all $t>0$ for $\mu^\omega$-almost every $\mathsf{x}\in M$ and $\nu$-almost every $\omega\in\Omega$.
\end{thrm}
Under the same assumptions, we can also prove an exponential law for the distribution of the return times. It is important to notice that this is the first paper where a quenched law is proved for the return times. Indeed, in \cite{RSV,RT,FFV17}, the law was only obtained for the hitting times. This is a significant difference with the deterministic setting where if a limiting distribution exists for the hitting times, then it also exists for the return times (and the other way around) \cite{HLV}.
\begin{thrm}\label{main.theorem1}
Let a random dynamical system satisfy the above requirements~(I)--(VI) where $\delta$ and $\lambda$ both decay super polynomially fast.
%Let $\rho_j\to 0$ be so that $\sum_j\rho_j^q<\infty$ for some positive 
%$q<d_0$.

Then
%$$
%\mu^\omega\!\left(y\in M: 
%\tau^\omega_{B_{\rho}(\mathsf{x})}(y)>\frac{t}{\mu(B_{\rho}(\mathsf{x}))}\right)
%\longrightarrow e^{-t}\qquad \mbox{ as }\rho\to0
%$$
%and
$$
\mu^\omega_{B_{\rho}(\mathsf{x})}\!\left(y\in M: 
\tau^\omega_{B_{\rho}(\mathsf{x})}(y)>\frac{t}{\mu(B_{\rho}(\mathsf{x}))}\right)
\longrightarrow e^{-t}\qquad \mbox{ as }\rho\to0
$$
for all $t>0$ for $\mu^\omega$-almost every $\mathsf{x}\in M$ and $\nu$-almost every $\omega\in\Omega$.
\end{thrm}
As can be observe in the next theorem, even if $\delta$ and $\lambda$ do not decay super polynomially fast, one can still obtain an exponential distribution assuming some technical conditions on the constants present in the hypothesis (I)--(VI). 

Let us denote $u_0$ the largest of the elements $\tilde u$ so that 
$\mu^\omega_{\gamma^u}(B_\rho(x))\le C_1\rho^{\tilde u}$ for all $\rho>0$ small enough
and for almost all $x\in \gamma^u$, every unstable leaf $\gamma^u$ and $\nu$-almost all $\omega$. By assumption (V), such element exists and is at least equal to $d_0$.

We have a version of Theorems \ref{main.theorem} and \ref{main.theorem1} for polynomial decay:
%%%%%%%%%%%%%%%%%%%%%%%%%%%%%%%%%%%%%%%%%%%%%
%%%%%%%%%%%%%%%%%%%%%%%%%%%%%%%%%%%%%%%%%%%%%
%%%%%%%%    THEOREM: MAIN THEOREM POLYNOM IAL
%%%%%%%%%%%%%%%%%%%%%%%%%%%%%%%%%%%%%%%%%%%%%

\begin{thrm}\label{main.theorem.polynomial}
Let a random dynamical system satisfy the above requirements~(I)--(VI).
Assume one of the following two conditions is satisfied.\\
(A) $\delta(n)=\mathcal{O}(n^{-\kappa})$ and $\lambda(k)=\mathcal{O}(k^{-p})$ 
decay polynomially with the respective rates $\kappa>1$ and $p>1$ satisfying  $\kappa\xi>1$, 
$\max\{\frac{d_1\beta}{\kappa\xi-1}, \left(\frac{\beta}{\xi}+d_1\right)\frac{1}{p} \} < \min\{1,u_0\}$ and $\gamma=\kappa u_0-2-\kappa'>1$.\\
%Let $\rho_j\to 0$ be a sequence so that $\sum_j\rho_j^q<\infty$ for some positive $q<d_0-\frac{d_0+d_0\xi + d_1\xi +\beta}{p\xi+1}$.
(B) $\delta$ decays super polynomially, $\lambda(k)=\mathcal{O}(k^{-p})$ 
decays polynomially and $\left(\frac{\beta}{\xi}+d_1\right)\frac{1}{p} < \min\{1,u_0\}$.
%Let $\rho_j\to 0$ be so that $\sum_j\rho_j^q<\infty$ for some positive 
%$q<d_0-\frac{d_0+d_0\xi + d_1\xi +\beta}{p\xi+1}$.
%\\
%(C) $\delta(n)$ and $\lambda$ both decay super polynomially fast.
%Let $\rho_j\to 0$ be so that $\sum_j\rho_j^q<\infty$ for some positive 
%$q<d_0$.

Then the conclusions of Theorems \ref{main.theorem} and \ref{main.theorem1} hold.
\end{thrm}
\noindent {\it Remark:} If $d$ is the dimension of the measure $\mu$ then 
$d_0<d<d_1$ can be chosen arbitrarily close to $d$. The assumptions in
case~(A) then simplify to
$\max(\frac{d\beta}{k\xi-1}, \frac{\beta/\xi+d}{p}) < 1\wedge u_0$.

The proof of the theorems is done in the next three sections. In Section~\ref{entry.times.distribution}
we prove that the limiting distribution is exponential (the convergence is realised
for $\mu^\omega$-almost every point $\mathsf{x}$) using a key proposition (Proposition~\ref{prop.short.returns}).  In Section~\ref{very.short.returns}, we prove the key proposition, i.e. we show that the measure of the set of points
whose neighbourhoods return to themselves within a very small number of iterates is small. In Section~\ref{return.times.distribution}
we then prove the limiting result for return times, an alternative proof is also given in Section~\ref{section.alternative}. Finally, in Section~\ref{example} 
we look at interval maps as an example to apply our main result.

Throughout the paper $C_0, C_1, \hdots$ and $\alpha, \beta, \hdots$ denote global constants
 while $c_0, c_1, \hdots$ are locally defined constants.

%%%%%%%%%%%%%%%%%%%%%%%%%%%%%%%%%%%%%%%%%%%%%
%%%%%%%%%%%%%%%%%%%%%%%%%%%%%%%%%%%%%%%%%%%%%
%%%%%%%%%%  SECTION: PROOF OF THE THEOREM FOR ENTRY TIMES
%\section{Entry times distribution}\label{entry.times.distribution}
\section{Very Short Returns}\label{very.short.returns}

For a ball $\ball(\mathsf{x})\subset M$ we define the counting function
\vspace{-0.2cm}
\begin{equation*}
Z^{\omega}_{\mathsf{x},\rho,t}(y)=\sum_{n=0}^{\floor{t/\mu(\ball(\mathsf{x}))}-1} {\ind_{\ball(\mathsf{x})}} \circ T_\omega^n(y)
\end{equation*}
which tracks the number of visits a trajectory of the point $y \in M$ makes to the ball $\ball(\mathsf{x})$
on an orbit segment of length $N=\floor{t/\mu(\ball(\mathsf{x}))}$, where $t$ is a positive
parameter.
Clearly $\tau^\omega_{\ball(\mathsf{x})}(y)>N$ exactly if 
$Z^{\omega}_{\mathsf{x},\rho,t}(y)=0$.

Let us put $J=\mathfrak{a}\abs{\log\rho}$ (with the number $\mathfrak{a}$ determined below)
and define the following counting function
for very short returns along the orbit segment:
$$
Y^\omega_{\mathsf{x},\rho,t}(y)
=\sum_{j=1}^N\ind_{\ball(\mathsf{x})\cap\{\tau^{\theta^j\omega}_{\ball(\mathsf{x})}<J\}}\circ T_\omega^j(y).
$$

In order to control the contribution made by the terms involving $Y^\omega_{\mathsf{x},\rho,t}$
we first have to consider the set of points that have very short returns. For that purpose, 
for a positive parameter $\mathfrak{a}$ let us define the set 
\begin{equation} \label{defM_rho,J}
\mathcal{V}_\rho^\omega(\mathfrak{a}) = \{\mathsf{x} \in M: \ball(\mathsf{x}) \cap T_\omega^{n}\ball(\mathsf{x}) \ne \varnothing \text{ for some } 1 \leq n < \mathfrak{a}\abs{\log\rho}\},
\end{equation}
where $\rho>0$. The set $\mathcal{V}_\rho^\omega$ represents the points within $M$ with very short return times with respect to the realisation $\omega$. The following of this section is dedicated to estimate the size of the set $\mathcal{V}_\rho^\omega$
and as a consequence to estimate the quantity $\mu^\omega(Y^\omega_{\mathsf{x},\rho,t})$.

%%%%%%%%%%%%%%%%%%%%%%%%%%%%%%%%%%%%%%%%%%%%%%%%
%%%%%%%%%%%%%%%%%%%%%%%%%%%%%%%%%%%%%%%%%%%%%%%%
%%%%%%%%%%%%%%%%%% SECTION VERY SHORT RETURNS

%\section{Very Short Returns}\label{very.short.returns}
Let us put $J=\mathfrak{a}\abs{\log\rho}$ where $\mathfrak{a} = (4 \log A)^{-1}$ with
$$
A =\sup_\omega 
\left(\norm{DT_\omega}_{\mathscr{L}^\infty} + \norm{DT_\omega^{-1}}_{\mathscr{L}^\infty}\right)
$$
($A\ge2$).
Then the set $\mathcal{V}_\rho \subset M$ consists of all 
$\mathsf{x} \in M$ for which $\ball(\mathsf{x}) \cap T_\omega^{n}\ball(\mathsf{x}) \ne \varnothing$
for some $1 \leq n < J$.

%%%%%%%%%%%%%%%%%%%%%%%%%%%%%%%%%%%%%%%%%%
%%%%%%%%%%%%%%%%%%%%%%%%%%%%%%%%%%%%%%%%%%
%%%%%%%%%%%     SUBSECTION: VERY SHORT RETURNS ESTIMATE
\subsection{Estimate on the measure of $\mathcal{V}_\rho$} \label{shortsMain}

Now we can show that the set of centres where small balls have very short returns is small. Even if we follow the proof of Proposition~5.1 of~\cite{HW14} which modelled after 
 Lemma~4.1 of~\cite{CC13}, we emphasize that this is not a direct adaptation, in particular in view of Lemma \ref{B.3} and its differences with the deterministic version.

%%%%%%%%%%%%%%%%%%%%%%%%%%%%%%%%%%%%%%%%%%%%
%%%%%%%%%%%%%%%%%%%%%%%%%%%%%%%%%%%%%%%%%%%%
%%%%%%%%%%%  PROPOSITION: VERY SHORT RETURNS
\begin{prop}\label{prop.short.returns}
There exist constants $C_{2}>0$ such that for all $\rho$ small enough and all $\omega$
and $\hat\omega$:
$$
\mu^{\hat{\omega}}(\mathcal{V}^\omega_\rho)
\le C_2\!\left(e^{-\mathfrak{c}\abs{\log \rho}^{1/2}}  +\delta(\mathfrak{ab}\abs{\log\rho})^{u_1}|\log\rho|^{\kappa'}\right)
$$
where  $u_1=u_0$ if $\delta$ decays superpolynomially and
 $u_1=u_0-\frac1\kappa$ if $\delta$ decays polynomially with power $\kappa$
 and $\mathfrak{b},\mathfrak{c}>0$ (recall that $\Theta(n)=\mathcal{O}(n^{\kappa'})$).
\end{prop}

\begin{proof} 
Let us note that since $T_\omega$ is a diffeomorphism one has
$$
\ball(\mathsf{x}) \cap T_\omega^{n}\ball(\mathsf{x}) \ne \varnothing \qquad \iff \qquad \ball(\mathsf{x}) \cap T_\omega^{-n}\ball(\mathsf{x}) \ne \varnothing.
$$
We partition $\mathcal{V}^\omega_\rho$ into level sets $\mathcal{N}^\omega_{\rho}(n)$ as follows
$$
\mathcal{V}^\omega_\rho = \{\mathsf{x} \in M: \ball(\mathsf{x}) \cap T_\omega^{-n}\ball(\mathsf{x}) \ne \varnothing \text{ for some } 1 \leq n < J\}
 = \bigcup_{n=1}^{J-1} \mathcal{N}^\omega_{\rho}(n)
 $$
 where
 $$
  \mathcal{N}^\omega_{\rho}(n) = \{\mathsf{x} \in M: \ball(\mathsf{x}) \cap T_\omega^{-n}\ball(\mathsf{x}) \ne \varnothing \}.
$$
The above union is split into two collections $\mathcal{V}_\rho^{\omega,1}  $ and $\mathcal{V}_\rho^{\omega,2} $, where
\begin{equation*}
\mathcal{V}_\rho^{\omega,1} = \bigcup_{n=1}^{\floor{\mathfrak{b} J}} \mathcal{N}^\omega_{\rho}(n) \quad \text{and} \quad \mathcal{V}_\rho^{\omega,2} = \bigcup_{n=\roof{\mathfrak{b} J}}^{J} \mathcal{N}^\omega_{\rho}(n).
\end{equation*}
and where the constant $\mathfrak{b} \in (0,1)$ will be chosen below.
In order to find the measure of the total set we will estimate the measures of the two parts separately.

\vspace{3mm}

%%%%%%%%%%%%%%%%%%%%%%%%%%%%%%%%          Part 1
\noindent {\bf (I) Estimate of $\mathcal{V}_\rho^{\omega,2}$}

\vspace{1mm}

\noindent We will derive a uniform estimate for the measure of the level sets $\mathcal{N}^\omega_{\rho}(n)$ when $n > \mathfrak{b} J$.
For  this purpose define 
$$
\tilde\omega=\omega_0\dots\omega_{n-1}{\hat{\omega}}_0{\hat{\omega}}_1\dots 
= \omega_0\dots\omega_{n-1}{\hat{\omega}}.
$$
We have $T_{\tilde{\omega}}^n = T_{\omega_{n-1}}\circ \cdots \circ T_{\omega_0} = T_\omega^n$. Also notice that $\theta^n\tilde{\omega} = {\hat{\omega}}$. Thus
$$
\mu^{\hat{\omega}}(\mathcal{N}^\omega_\rho(n))=\mu^{\tilde{\omega}}(T_{\tilde{\omega}}^{-n}\mathcal{N}^\omega_{\rho}(n))
\le\sum_{\zeta} \mu^{\tilde{\omega}}(T_\omega^{-n}\mathcal{N}^\omega_{\rho}(n) \cap \zeta)
$$
where the sum is over all $n$-cylinders $\zeta$.
We will consider each of the measures $\mu^{\tilde{\omega}}(T_\omega^{-n}\mathcal{N}^\omega_{\rho}(n) \cap \zeta)$ separately by using the product form of the measures $\mu^{\hat{\omega}}$.
By distortion of the Jacobian we obtain
\begin{eqnarray}
\mu^{\tilde{\omega}}_{\gamma^u}(T_\omega^{-n}\mathcal{N}^\omega_{\rho}(n) \cap \zeta)
&=& \frac{\mu^{\tilde{\omega}}_{\gamma^u}(T_\omega^{-n}\mathcal{N}^\omega_{\rho}(n) \cap \zeta)}
{\mu^{\tilde{\omega}}_{\gamma^u}(\zeta)} \, \mu_{\gamma^u}^{\tilde{\omega}}(\zeta)\notag\\
&\leq &\Theta(n)  \,\frac{\mu^{\hat{\omega}}_{\hat\gamma^u}(T_\omega^{n}(T_\omega^{-n}\mathcal{N}^\omega_{\rho}(n) \cap \zeta))}{\mu^{\hat{\omega}}_{\hat\gamma^u}(T_\omega^n\zeta)} \,\mu_{\gamma^u}^{\tilde{\omega}}(\zeta),  \label{level_summand}
\end{eqnarray}
where, as before, $\hat\gamma^u=\gamma^u(T_\omega^nx)$ for $x\in\zeta\cap\gamma^u$.
We estimate the numerator by finding a bound for the diameter of the set. Let the points $x$ and $z$ in
$T_\omega^{-n}\mathcal{N}^\omega_{\rho}(n)$ be such that 
$ x, z \in T_\omega^{-n}\mathcal{N}^\omega_{\rho}(n) \cap \zeta\cap\gamma^u$
for an unstable leaf $\gamma^u$. 

Note that $T_\omega^n x , T_\omega^n z \in \mathcal{N}^\omega_{\rho}(n)$, there exists $y\in B_\rho(T^n_\omega x)$ such that $ T^n_\omega y\in B_\rho(T^n_\omega x)$
(as $\ball(T_\omega^nx)\cap T_\omega^{-n}\ball(T_\omega^nx)\not=\emptyset$), thus
$$
d(T_\omega^n x, x)
\leq d(T_\omega^n x, T_\omega^n y )+d(T_\omega^n y, y)+d(y, x)
\leq \rho+2\rho+A^n d(T_\omega^n x, T_\omega^n y )
\leq (3+A^n)\rho.
$$
Hence as $y\in B_{A^n\rho}(x)$:
$$
d(T_\omega^n x, T_\omega^n z) 
\leq d(T_\omega^n x, x) + d(x,z) + d(z, T_\omega^n z)
 \leq 6 A^n \rho + d(x,z).
$$ 
We have
$$
d(x, z) \leq \diam \zeta < \delta(n)
$$
by assumption. Therefore
$$
d(T_\omega^nx, T_\omega^nz)  
\leq 6A^n \rho + d(x,z) 
 \leq 6\, A^{n} \rho + \delta(n)
$$
If we choose $\mathfrak{a}>0$ so that $\mathfrak{a}<\frac{1}{2\log A}$ then
$A^n\rho< e^{-\frac12\abs{\log\rho}^{1/2}}$. If $n\ge \mathfrak{b}\abs{\log\rho}$ for 
some $\mathfrak{b}\in(0,\mathfrak{a})$ then
$$
d(T_\omega^nx, T_\omega^nz)  \leq c_1(e^{-\mathfrak{c}'\abs{\log\rho}^{1/2}}+\delta(n))
$$
for some constant $c_1$ where $\mathfrak{c}'=\min(\frac12,\sqrt\mathfrak{b})$.
Taking the supremum over all points $x$ and $z$ yields
$$
\abs{T_\omega^n (T_\omega^{-n}\mathcal{N}^\omega_{\rho}(n) \cap \zeta\cap\gamma^u)} 
\leq c_1(e^{-\mathfrak{c}'\abs{\log\rho}^{1/2}}+\delta(n)).
$$
By assumption (V) on the relationship between the measure and the metric
$$
\mu^{\hat{\omega}}_{\hat\gamma^u}(T_\omega^n (T_\omega^{-n}\mathcal{N}^\omega_{\rho}(n) \cap \zeta))
\leq c_2(e^{-u_0\mathfrak{c}'\abs{\log \rho}^{1/2}}+\delta(n)^{u_0})
$$
%which implies by the product structure of $m$ that 
%$$
%\mu^\omega_{\hat\gamma^u}(T_\omega^n (T_\omega^{-n}\mathcal{N}_{\rho}(n) \cap \zeta)) 
%\leq c_3(e^{-{u_0}\mathfrak{c}'\abs{\log \rho}^{1/2}}+\delta(n)^{u_0})
%$$
Incorporating the estimate into~\eqref{level_summand} yields
$$
\mu^{\tilde{\omega}}_{\gamma^u}(T_\omega^{-n}\mathcal{N}^\omega_{\rho}(n) \cap \zeta)
\leq c_4\Theta(n)\, (e^{-{u_0}\mathfrak{c}'\abs{\log \rho}^{1/2}}+\delta(n)^{u_0})
 \mu^{\tilde{\omega}}(\zeta),
$$
for some $c_4$.
Integrating over $d\upsilon(\gamma^u)$ and summing over $\zeta$ yields
$$
\mu^{\hat{\omega}}(\mathcal{N}^\omega_{\rho}(n))
\le c_4\Theta(n) \, (e^{-{u_0}\mathfrak{c}'\abs{\log \rho}^{1/2}}+\delta(n)^{u_0})
\sum_\zeta \mu^{\tilde{\omega}}(\zeta) 
\le c_5\Theta(n) \, (e^{-{u_0}\mathfrak{c}'\abs{\log \rho}^{1/2}}+\delta(n)^{u_0})
$$
as $\sum_\zeta \mu^{\tilde{\omega}}(\zeta)=\mathcal{O}(1)$.
Consequently
\begin{eqnarray}
\mu^{\hat{\omega}}(\mathcal{V}_\rho^{\omega,2}) 
& \le& \sum_{n=\roof{\mathfrak{b} J}}^{J} \mu^{\hat{\omega}}(\mathcal{N}^\omega_{\rho}(n))\notag\\
 &\le& c_5\Theta(J) \, Je^{-{u_0}\mathfrak{c}'\abs{\log \rho}^{1/2}} 
 + \sum\limits_{n=\roof{\mathfrak{b} J}}^{J}\Theta(n)\delta(n)^{u_0}\notag\\
 &\le& c_6(e^{-\mathfrak{c}''\abs{\log \rho}^{1/2}}
 +\delta(\mathfrak{ab}\abs{\log\rho})^{u_1}|\log\rho|^{\kappa'})\label{Th2refPt2}
\end{eqnarray}
for some constant $\mathfrak{c}''>0$ (and $\rho$ small enough) as $J=\lfloor\mathfrak{a}\abs{\log\rho}\rfloor$. Here $u_1\le u_0$
is so that $\sum_{n=n_0}^\infty\delta(n)^{u_0}\le c_7\delta(n_0)^{u_1}$
for some constant $c_7$.

\vspace{3mm}

%%%%%%%%%%%%%%%%%%%%%%%%%%%%%%%%%%%%%%    Part 2
\noindent {\bf (II) Estimate of $\mathcal{V}_\rho^{\omega,1}$}

\vspace{1mm}

We will need the following randomised version of Lemma~B.3 from~\cite{CC13}.

\begin{lem}\label{B.3}
Put $s_p = 2^p \, \frac{A^{n \, 2^p}-1}{A^n-1}$. Then  for every $p,k$ integers, $\rho>0$ and $\omega$
there exists an $\tilde{\omega}$ so that
$$
\left\{\mathsf{x} \in M: B_{\rho}(\mathsf{x}) \cap T_\omega^{k}B_{\rho}(\mathsf{x}) \ne \varnothing\right\}
\subset
\left\{\mathsf{x} \in M: B_{s_p \rho}(\mathsf{x}) \cap T_{\tilde{\omega}}^{k2^p}B_{s_p \rho}(\mathsf{x}) \ne \varnothing\right\}.
$$
\end{lem}
\noindent {\bf Proof.} 
Consider the case $p=1$. Let $x$ such that $\ball(\mathsf{x}) \cap T_\omega^{k}\ball(\mathsf{x}) \ne \varnothing$. This implies that there exist $z\in \ball(\mathsf{x}) \cap T_\omega^{-k}\ball(\mathsf{x})$. For any $u\in T_\omega^{k}\ball(\mathsf{x})$, there exist $v\in \ball(\mathsf{x})$ such that $T_\omega^k v=u$, thus
\[d(u,x)\leq d(u, T_\omega^kz)+d(T_\omega^kz,x)\leq d(T_\omega^k v,T_\omega^kz)+2\rho\leq (2A^k+2)\rho.\]
Therefore, $T_\omega^{k}\ball(\mathsf{x})\subset B_{(2A^k+2)\rho}(\mathsf{x})$.

One can observe that if $\ball(\mathsf{x}) \cap T_\omega^{k}\ball(\mathsf{x}) \ne \varnothing$ then $T_\omega^{k}\left(\ball(\mathsf{x}) \cap T_\omega^{k}\ball(\mathsf{x})\right) \ne \varnothing$ thus $T_\omega^{k}\ball(\mathsf{x}) \cap  T_\omega^{k}(T_\omega^{k}\ball(\mathsf{x})) \ne \varnothing$ and therefore $B_{(2A^k+2)\rho}(\mathsf{x}) \cap  T_\omega^{k}(T_\omega^{k}B_{(2A^k+2)\rho}(\mathsf{x})) \ne \varnothing$. Finally, this gives us
\[\{\mathsf{x} \in M: \ball(\mathsf{x}) \cap T_\omega^{k}\ball(\mathsf{x}) \ne \varnothing \}\subset \{\mathsf{x} \in M: B_{(2A^k+2)\rho}(\mathsf{x}) \cap T_\omega^{k}(T_\omega^{k}B_{(2A^k+2)\rho}(\mathsf{x})) \ne \varnothing \}.\]
If we put $\tilde\omega=\omega_0\dots\omega_{k-1}\omega_0\omega_1\dots$ then $T_\omega^{k}\circ T_\omega^{k}=T^{2k}_{\tilde\omega}$. This proves the case $p=1$. The general case is shown similarly.
\qed

\vspace{3mm} 

%\noindent The lemma thus shows that 
% $\mathcal{N}^\omega_{\rho}(n) \subset \mathcal{N}^{\tilde\omega}_{s_p \rho}(2^p n)$ 
% for some $\tilde\omega$ and consequently we only need to estimate 
% $\mu_\omega(\mathcal{N}^{\tilde\omega}_{s_p \rho}(2^p n))$.

Let us now consider the case $1 \leq n \leq \floor{\mathfrak{b}J}$ and let as in Lemma~\ref{B.3}
$s_p = 2^p \, \frac{A^{n \, 2^p}-1}{A^n-1}$.
Hence by Lemma~\ref{B.3} one has 
$\mathcal{N}^\omega_{\rho}(n) \subset \mathcal{N}^{\tilde{\omega}}_{s_p \rho}(2^p n)$,
where $\tilde\omega=\tilde{\omega}(n)$ depends on $n$,
for any $p \geq 1$, and in particular for $p(n) = \floor{\log \mathfrak{b}J - \lg n} + 1$.
Therefore
$$
\bigcup_{n=1}^{\floor{\mathfrak{b}J}} \mathcal{N}^\omega_{\rho}(n) \subset \bigcup_{n=1}^{\floor{\mathfrak{b}J}} \mathcal{N}^{\tilde\omega}_{s_{p(n)} \rho}(2^{p(n)} n).
$$
Now define
\begin{equation*}
n' = n 2^{p(n)} \qquad \text{ and } \qquad \rho' = s_{p(n)} \rho.
\end{equation*}
A direct computation shows that $1 \leq n \leq \floor{\mathfrak{b}J}$ implies $\roof{\mathfrak{b}J} \leq n' \leq 2 \mathfrak{b}J$ and so
$$
\mathcal{V}_\rho^{\omega,1} = \bigcup_{n=1}^{\floor{\mathfrak{b}J}} \mathcal{N}^\omega_{\rho}(n) \subset \bigcup_{n=1}^{\floor{\mathfrak{b}J}} \mathcal{N}^{\tilde\omega}_{s_{p(n)} \rho}(2^{p(n)} n) \subset \bigcup_{n'=\roof{\mathfrak{b}J}}^{2 \mathfrak{b}J} \mathcal{N}^{\tilde\omega}_{\rho'}(n').
$$
Therefore to estimate the measure of $\mathcal{V}_\rho^1$ it suffices to find a bound for
$\mu^{\hat{\omega}}(\mathcal{N}^{\tilde\omega}_{\rho'}(n'))$ when $n' \geq \mathfrak{b}J$. This is accomplished by using
an argument analogous to the first part of the proof. 
 Notice that since $\tilde{\omega}=(\omega_0\dots\omega_{k-1})^{2^p-1}\omega $ we get that $$
 T^{n'}_{\tilde{\omega}} = \left(T_\omega^n\right)^{2^p}.
 $$
Define 
$\omega' = \omega_0\dots\omega_{k-1} \tilde\omega = (\omega_0\dots\omega_{k-1})^{2^p}\omega$
and 
$\hat{\omega}' = (\omega_0\dots\omega_{k-1})^{2^p}\hat{\omega}$.
Notice that $\theta^{2^pn}\omega' = \theta^{(2^p-1)n}\tilde{\omega} = \omega$, we get 
$$
T^{n'}_{\omega'} =  \left(T_\omega^n\right)^{2^p}=T^{n'}_{\tilde{\omega}}.
$$
As a result,
$$
\mathcal{N}^{\tilde\omega}_{\rho'}(n') = \{x: B_{\rho'} \cap T^{n'}_{\tilde{\omega}}B_\rho \neq \emptyset\}=\{x: B_{\rho'} \cap T^{n'}_{\omega'}B_\rho \neq \emptyset\} = \mathcal{N}^{\omega'}_{\rho'}(n').
$$
Similarly to the part~(I),
\begin{eqnarray*}
\mu^{\hat{\omega}}(\mathcal{N}^{\tilde\omega}_{\rho'}(n'))
&=&\mu^{\hat{\omega}}(\mathcal{N}^{\omega'}_{\rho'}(n'))\\
&=&\mu^{\hat{\omega}'}_{\gamma^u}(T_{\omega'}^{-n'}\mathcal{N}^{\omega'}_{\rho'}(n') \cap \zeta)\\
&=& \frac{\mu^{\hat{\omega}'}_{\gamma^u}(T_{\omega'}^{-n'}\mathcal{N}^{\omega'}_{\rho'}(n') \cap \zeta)}
{\mu^{\hat{\omega}'}_{\gamma^u}(\zeta)} \, \mu_{\gamma^u}^{\hat{\omega}'}(\zeta)\notag\\
&\le& \Theta(n')  \,\frac{\mu^{\hat\omega}_{\hat\gamma^u}(T_{\omega'}^{n'}(T_{\omega'}^{-n'}\mathcal{N}^{\omega'}_{\rho'}(n') \cap \zeta))}{\mu^{\hat\omega}_{\hat\gamma^u}(T_{\omega'}^{n'}\zeta)} \,
\mu_{\gamma^u}^{\hat{\omega}'}(\zeta).
\end{eqnarray*}
To estimate the measure of the numerator we follow the proof of Proposition~\ref{prop.short.returns} and replace all the $n$ with $n'$
 and $\rho$ with $\rho'$.  We get for $\mathfrak{b}< 1/3$
$$
\diam(T_{\omega'}^{n'}(T_{\omega'}^{-n'}\mathcal{N}^{\omega'}_{\rho'}(n') \cap \zeta))
\leq c_1(e^{-\mathfrak{c}'\abs{\log\rho'}^{1/2}}+\delta(n')).
$$
Therefore
$$
\mu^{\hat{\omega}}(\mathcal{N}^{\tilde\omega}_{\rho'}(n'))
 \leq c_5\Theta(n')(e^{-{u_0}\abs{\log \rho'}^{1/2}}  +\delta(n')^{u_0}).
$$
Since $\rho'=s_p\rho$ and $\mathfrak{b}<\mathfrak{a}=\frac{1}{4\log A}$, we have
\begin{eqnarray*}
\rho'&\leq&A^{2n2^p}\rho=A^{2n'}\rho\\
&\leq&A^{4\mathfrak{b}J}=A^{4\mathfrak{ab}\abs{\log \rho}}\rho\\
&\leq&A^{\frac{4}{16\log A}\abs{\log \rho}}\rho=\rho^{3/4}
\end{eqnarray*}
which gives us
$$
\mu^{\hat{\omega}}(\mathcal{N}^{\tilde\omega}_{\rho'}(n'))
 \leq c_5\Theta(n')(e^{-{u_0}\abs{\log \rho^{3/4}}^{1/2}}  +\delta(n')^{u_0}).
$$
Thus, we obtain an estimate similar to~\eqref{Th2refPt2}:
$$
\mu^{\hat{\omega}}(\mathcal{V}_\rho^{\omega,1}) 
\leq \sum_{n'=\roof{\mathfrak{b}J}}^{2 \mathfrak{b}J} \mu^{\hat{\omega}}(\mathcal{N}^{\tilde\omega}_{\rho'}(n'))
 \le c_8(e^{-\mathfrak{c}\abs{\log \rho}^{1/2}}  
 +\delta(\mathfrak{ab}\abs{\log\rho})^{u_1}|\log\rho|^{\kappa'}).
$$
for some $\mathfrak{c}\in(0,\frac{3}{4}u_0)$.

\vspace{3mm}

%%%%%%%%%%%%%%%%%%%%%%%%%%%%%%%%%%%%%%%% Part 3
\noindent {\bf (III) Final estimate}

\vspace{1mm}

\noindent Overall we obtain for all $\rho$ sufficiently small
$$
\mu^{\hat{\omega}}(\mathcal{V}^\omega_\rho) 
\leq \mu^{\hat{\omega}}(\mathcal{V}_\rho^{\omega,1})+\mu^{\hat{\omega}}(\mathcal{V}_\rho^{\omega,2})
\le C_{3}\, (e^{-\mathfrak{c}\abs{\log \rho}^{1/2}}  
+\delta(\mathfrak{ab}\abs{\log\rho})^{u_1}|\log\rho|^{\kappa'}),
$$
for some $C_2$.
\end{proof}

%%%%%%%%%%%%%%%%%%%%%%%%%%%%%%%%%%%%%%%%%%
%%%%%%%%%%%%%%%%%%%%%%%%%%%%%%%%%%%%%%%%%%
%%%%%%%%%%%     SUBSECTION:  INTEGRAL OF VERY SHORT RETURNS
\subsection{Estimate of $\mu^\omega(Y^\omega)$}
Now we are in a position to estimate the $\mu^\omega$-measure of the function $Y^\omega_{\mathsf{x},\rho,t}$.
 
\begin{lem}\label{integral.very.short.returns}
Put $\gamma= \kappa u_0-2-\kappa'$. Then for any $\gamma'\in(0,\gamma)$  there exists a set 
$\mathcal{B}^\omega_\rho$ so that
 $$
 \mu^\omega(Y^\omega_{\mathsf{x},\rho,t})\le \abs{\log\rho}^{-\gamma'}.
 $$
 for all $\mathsf{x}\not\in\mathcal{B}^\omega_\rho$ and
$$
\mu^\omega(\mathcal{B}_\rho^\omega)\lesssim\abs{\log\rho}^{-(\gamma-\gamma')}.
$$
 \end{lem}

\begin{proof} In order to estimate the term $\mu^\omega(Y^\omega_{\mathsf{x},\rho})$ let us put 
$N_\rho(\mathsf{x})=\floor{t/\mu(\ball(\mathsf{x}))}$.
Observe that if $\ball(\mathsf{x})\cap\{\tau_{\ball(\mathsf{x})}^{\theta^j\omega}<J\}\not=\varnothing$ then
there exist $y\in\ball(\mathsf{x})$ and $k<j$ so that $T_{\theta^j\omega}^ky\in\ball(\mathsf{x})$ which implies
 that $\ball(\mathsf{x})\cap T_{\theta^j\omega}^k\ball(\mathsf{x})\not=\varnothing$ and therefore
 $\mathsf{x}\in\mathcal{V}_\rho^{\theta^j\omega}$. 
 Hence  $\mathsf{x}\not\in\mathcal{V}_\rho^{\theta^j\omega}$ implies
 $\ball(\mathsf{x})\cap\{\tau_{\ball(\mathsf{x})}^{\theta^j\omega}<J\}=\varnothing$.

Put $W^\omega_\rho(\mathsf{x})
=\sum_{j=1}^{N_{\rho}(\mathsf{x})}\ind_{\mathcal{V}^{\theta^j\omega}_{\rho}}(\mathsf{x})$
and $q^\omega_\rho(\mathsf{x})=\frac{W^\omega_\rho(\mathsf{x})}{N_{\rho}(\mathsf{x})}$.
Let $M_k=\{\mathsf{x}\in M: N_{\rho}(\mathsf{x})=k\}$ and 
put $a_{j,k}=\mu^\omega(\mathcal{V}^{\theta^j\omega}_\rho\cap M_k)$. Observe that, by Assumption~(V) , $\sup_{\mathsf{x}}N^\omega_{\rho}(\mathsf{x})$ is bounded above by $\hat{N}=c_1t\rho^{-d_1}$ for some constant $c_1$.
Then 
$$
Q^\omega_\rho
:=\int_Mq^\omega_\rho(\mathsf{x})\,d\mu^\omega(\mathsf{x})
=\sum_{k=1}^{\hat{N}}\frac1k\sum_{j=1}^ka_{j,k}
\le\sum_{j=1}^{\hat{N}}\frac1{j}\sum_{k=1}^{{\hat{N}}}a_{j,k}.
$$

Since by Proposition~\ref{prop.short.returns} $\mu^\omega(\mathcal{V}^{\theta^j\omega}_\rho)=\sum_{k=1}^{{\hat{N}}}a_{j,k}\lesssim|\log\rho|^{-\gamma-1}$,
where $\gamma=\kappa u_0-2-\kappa'$, we thus obtain
$$
Q^\omega_\rho
\lesssim|\log\rho|^{-\gamma-1}\sum_{j=1}^{\hat{N}}\frac1j
\lesssim|\log\rho|^{-\gamma}.
$$
Now define
$$
\mathcal{B}^\omega_\rho=\left\{\mathsf{x}\in M:q^\omega_\rho(\mathsf{x})>|\log\rho|^{-\gamma'}\right\}
$$
for  $\gamma'\in(0,\gamma)$.
By Markov's inequality:
$$
\mu^\omega(\mathcal{B}^\omega_\rho)\le Q^\omega_\rho|\log\rho|^{\gamma'}
\lesssim|\log\rho|^{-\gamma''},
$$
where $\gamma''=\gamma-\gamma'$.

If $\mathsf{x}\not\in\mathcal{B}^\omega_{\rho}$ then 
$q^\omega_{\rho}(\mathsf{x})\lesssim|\log\rho|^{-\gamma'}$
and $W^\omega_{\rho}(\mathsf{x})\lesssim\frac{N_\rho(\mathsf{x})}{|\log\rho|^{\gamma'}}$.
Consequently there exists an  index set 
$\mathcal{I}^\omega_{\mathsf{x}}\subset\{1,2,\dots,N_\rho(\mathsf{x})\}$
so that $|\mathcal{I}^\omega_{\mathsf{x}}|\lesssim N_\rho(\mathsf{x})|\log\rho|^{-\gamma'}$
and
$$
\begin{cases}
\mathsf{x}\in\mathcal{V}^{\theta^j\omega}_{\rho}&\forall j\in\mathcal{I}^\omega_{\mathsf{x}}\\
\mathsf{x}\not\in\mathcal{V}^{\theta^j\omega}_{\rho}&\forall  
j\in\{1,\dots,N_\rho(\mathsf{x})\}\setminus\mathcal{I}^\omega_{\mathsf{x}}.
\end{cases}
$$
Since $\ball(\mathsf{x})\cap\{\tau^{\theta^j\omega}_{\ball(\mathsf{x})}<J\}=\varnothing$
for all $j\in\{1,\dots,N_\rho(\mathsf{x})\}\setminus\mathcal{I}^\omega_{\mathsf{x}}$
we finally get by Assumption~(V)
\begin{eqnarray*}
\mu^\omega(Y^\omega_{\mathsf{x},\rho})
&\le&\mu^\omega\!\left(\sum_{j\in\mathcal{I}^\omega_{\mathsf{x}}}
\ind_{\ball(\mathsf{x})\cap\{\tau^{\theta^j\omega}_{\ball(\mathsf{x})}<J\}}\circ T_\omega^j\right)\\
&\le&|\mathcal{I}^\omega_{\mathsf{x}}|K\mu(\ball(\mathsf{x}))\\
&\lesssim& q^\omega_{4\rho}(\mathsf{x})N_\rho(\mathsf{x})K\mu(\ball(\mathsf{x}))\\
&\lesssim&\frac{t}{|\log\rho|^{\gamma'}}
\end{eqnarray*}
\end{proof}

\section{Principal part}

In this section we will look at the case when all returns within the observation time $N$
are longer than $J$. In fact we want to show that following result which compares
the hitting times distribution to that of independent random variables.

%%%%%%%%%%%%%%%%%%%%%%%%%%%%%%%%%%%%%%%%
 %%%%%%%%%%%%%%%%%%%%%%%%%%%%%%%%%%%%%%%%
 %%%%%%%      PROPOSITION: MAIN ESTIMATE
 %%%%%%%%%%%%%%%%%%%%%%%%%%%%%%%%%%%%%%%%
\begin{prop}\label{main.proposition}
Under the assumptions of Theorem~\ref{main.theorem} put $u_1=u_0$ if
$\delta$ decays superpolynomially and $u_1=u_0-\frac1\kappa$ if 
$\delta(n)=\mathcal{O}(n^{-\kappa})$.

Then there exist a positive $\epsilon$
and a constant $C_3$ so that 
$$
\left|\mu^\omega(\tau^\omega_{\ball}>N)-\prod_{j=1}^N(1-\mu^{\theta^j\omega}(\ball))\right|
\le C_3\!\left(\rho^\epsilon+(\delta(J)^{u_1}+\rho^\epsilon)\sum_{j=1}^N\mu^{\theta^j\omega}(\ball)\right)
+\mu^\omega(Y^\omega_{\mathsf{x},\rho,t})
$$
for all balls $\ball$.% centred at $\mathsf{x}\not\in\mathcal{V}^\omega_\rho(\mathfrak{a})$, where 
%$$
%\mu^\omega(\mathcal{V}^\omega_\rho)
%\le C_2(e^{-\mathfrak{c}\abs{\log \rho}^{1/2}}  +\delta(\mathfrak{ab}\abs{\log\rho})^{u_1}|\log\rho|^{\kappa'})
%$$
\end{prop}

%%%%%%%%%%%%%%%%%%%%%%%%%%%%%%%%%%%%%%%%
 %%%%%%%%%%%%%%%%%%%%%%%%%%%%%%%%%%%%%%%%
 %%%%%%%      PROOF OF THE MAIN PROPOSITION
 %%%%%%%%%%%%%%%%%%%%%%%%%%%%%%%%%%%%%%%%
\begin{proof}
We proceed as in~\cite{RSV} and note that
$$
\left|\mu^\omega(\tau^\omega_{\ball}>N)-\prod_{j=1}^N(1-\mu^{\theta^j\omega}(\ball))\right|
\le\sum_{j=1}^N\epsilon_{\theta^j\omega}(\ball)\prod_{k=1}^{j-1}(1-\mu^{\theta^k\omega}(\ball))
$$
where 
\begin{equation}\label{defeps}
\epsilon_{\omega}(\ball)
=\sup_{k\ge 1} \left|\mu^\omega(\tau_{\ball}^\omega>k)\,\mu^\omega(\ball) - \mu^\omega(\ball\cap \{\tau_{\ball}^\omega>k\})\right|.
\end{equation}

We now split the error term on the RHS into three parts using the 
fact that 
\begin{eqnarray*}
\epsilon_{\theta^j\omega}
&\le&  \sup_{k\ge1}\left|\mu^{\theta^j\omega}(\ball \cap T_{\theta^j\omega}^{-\Delta}\{\tau_{\ball}^{\theta^{j+\Delta}\omega}\ge k\}) 
- \mu^{\theta^j\omega}(\ball) \, \mu^{\theta^{j+\Delta}\omega}(\{\tau_{\ball}^{\theta^{j+\Delta}\omega}\ge k\})\right|\\
&&\hspace{4cm}+\mu^{\theta^j\omega}(\ball\cap\{\tau^{\theta^j\omega}_{\ball}\le\Delta\})
+\mu^{\theta^j\omega}(\ball)\mu^{\theta^j\omega}(\tau^{\theta^j\omega}_{\ball}\le\Delta).
\end{eqnarray*}
Thus
\begin{equation}\label{estsumeps}
\left|\mu^\omega(\tau^\omega_{\ball}>N)-\prod_{j=1}^N(1-\mu^{\theta^j\omega}(\ball))\right|
\le\sum_{j=1}^N\epsilon_{\theta^j\omega}(\ball)
=\mathcal{R}
= \mathcal{R}_1+\mathcal{R}_2+\mathcal{R}_3
\end{equation}
where 
\begin{eqnarray*}
\mathcal{R}_1&=&\sum_{j=1}^N \sup_{k\ge1}\left|\mu^{\theta^j\omega}(\ball \cap T_{\theta^j\omega}^{-\Delta}\{\tau_{\ball}^{\theta^{j+\Delta}\omega}\ge k\}) 
- \mu^{\theta^j\omega}(\ball) \, \mu^{\theta^{j+\Delta}\omega}(\{\tau_{\ball}^{\theta^{j+\Delta}\omega}\ge k\})\right|\\
\mathcal{R}_2&=&\sum_{j=1}^N\mu^{\theta^j\omega}(\ball\cap\{\tau^{\theta^j\omega}_{\ball}\le\Delta\})
\\
\mathcal{R}_3&=&\sum_{j=1}^N\mu^{\theta^j\omega}(\ball)\mu^{\theta^j\omega}(\tau^{\theta^j\omega}_{\ball}\le\Delta).
\end{eqnarray*}

We now estimate the three terms individually.

%%%%%%%%%%%%%%%%%%%%%%%%%%%%%%%%%%%%%%%%%%%
%%%%%%%%%%%%%%%%%%%%%%%%%%%%%%%%%%%%%%%%%%%
%%%%%%%%%%%%%%%%% ESTIMATING R1
\subsection{Estimating $\mathcal{R}_1$} \label{estimate.r1}
%$$
%\|g'_\varepsilon\|
%$$

We estimate the principal term by
$$
\mathcal{R}_1 =N \sup_\omega \sup_{k\ge1}\left| \mu^\omega(\ball \cap T_\omega^{-\Delta}S_k) 
- \mu^\omega(\ball) \, \mu^{\theta^{\Delta}\omega}(S_k) \right|.
$$
where we put $S_k=S_k(\Delta)=\{y:\tau_{\ball}^{\theta^{\Delta}\omega}(y)\ge k\}$.
We now use the decay of correlations from Assumption~(II) to obtain an estimate. 
Approximate $\ind_{\ball}$ by Lipschitz functions
from above and below as follows:
$$
\phi(x) =
\begin{cases}
1 & \text{on $\ball$} \\
0 & \text{outside $B_{\rho + \delta \rho}$}
\end{cases}
\hspace{0.7cm} \text{and} \hspace{0.7cm}
\tilde{\phi}(x) =
\begin{cases}
1 & \text{on $B_{\rho - \delta \rho}$} \\
0 & \text{outside $\ball$}
\end{cases}
$$
with both functions linear within the annuli. The Lipschitz norms of both $\phi$ and $\tilde{\phi}$ are equal to $1/\delta\rho$ and $\tilde{\phi} \leq \ind_{\ball} \leq \phi$.
We obtain
\begin{align*}
&\mu^\omega(\ball \cap T_\omega^{-\Delta}S_k) 
- \mu^\omega(\ball) \, \mu^{\theta^{\Delta}\omega}(S_k)\hspace{-0cm}\\
& \leq \int_M \phi (\ind_{S_k}\circ T_\omega^{\Delta}) \, d\mu^\omega 
- \int_M \ind_{\ball} \, d\mu^\omega \, \int_M \ind_{S_k} \, d\mu^{\theta^{\Delta}\omega} \\[0.2cm]
& =X+Y
\end{align*}
where
\begin{align*}
X&=\left(\int_M \phi \, d\mu^\omega - \int_M \ind_{\ball} \, d\mu^\omega \right)
 \int_M \ind_{S_k} \, d\mu^{\theta^\Delta\omega}\\
Y&=\int_M \phi \; (\ind_{S_k}\circ T_\omega^\Delta ) \, d\mu^\omega - \int_M \phi \, d\mu^\omega \, \int_M \ind_{S_k} \, d\mu^{\theta^\Delta\omega} .
\end{align*}
The two terms $X$ and $Y$ are estimated separately.
The first term is estimated as follows:
$$
X \leq\int_M \ind_{S_k} \, d\mu^{\theta^\Delta\omega} \, \int_M (\phi - \ind_{\ball}) \, d\mu^\omega
 \leq \mu^\omega(B_{\rho + \delta \rho} \setminus \ball).
$$
In order to estimate the second term $Y$ we use the decay of correlations and
have to approximate $\ind_{S_k}$ by a function which is constant on local stable leaves.
For that purpose put
$$
\mathcal{S}_n
=\bigcup_{\substack{\gamma^s\\T_\omega^n\gamma^s\subset B_\rho}}T_\omega^n\gamma^s,
\hspace{6mm}
\partial\mathcal{S}_n
=\bigcup_{\substack{\gamma^s\\T_\omega^n\gamma^s\cap B_\rho\not=\varnothing}}T_\omega^n\gamma^s
$$
and
$$
\mathscr{S}_\Delta^{N-j}=\bigcup_{n=\Delta}^{N-j}\mathcal{S}_n,
\hspace{6mm}
\partial\mathscr{S}_\Delta^{n-j}=\bigcup_{n=\Delta}^{N-j}\partial\mathcal{S}_n.
$$
The set
$$
\mathscr{S}_\Delta^{N-j}(k)=S_k\cap\mathscr{S}_\Delta^{N-j}
$$
is then a union of local stable leaves. This follows from the fact that by construction
$T^n y\in B_\rho$ if and only if $T^n\gamma^s(y)\subset B_\rho$.
We also have
$S_k\subset\tilde{\mathscr{S}}_\Delta^{N-j}(k)$
where the set $\tilde{\mathscr{S}}_\Delta^{N-j}(k)=\mathscr{S}_\Delta^{N-j}(k)\cup\partial\mathscr{S}_\Delta^{N-j}$
is a union of local stable leaves.

Denote by  $\psi_\Delta^{N-j}$ the indicator function of $\mathscr{S}_\Delta^{N-j}(k)$
and by $\tilde\psi_\Delta^{N-j}$ the indicator function of
$\tilde{\mathscr{S}}_\Delta^{N-j}(k)$. Then $\psi_\Delta^{N-j}$ and $\tilde\psi_\Delta^{N-j}$
are constant on local stable leaves and satisfy
$$
\psi_\Delta^{N-j}\le\ind_{S_k}\le\tilde\psi_\Delta^{N-j}.
$$
Since $\{y:\psi_\Delta^{N-j}(y)\not=\tilde\psi_\Delta^{N-j}(y)\}\subset\partial\mathscr{S}_\Delta^{N-j}$
we need to estimate the measure of $\partial\mathscr{S}_\Delta^{N-j}$.

 By the contraction property
  $\mbox{diam}(T_\omega^n\gamma^s(y))\le\delta(n)$ and consequently
  $$
  \bigcup_{\substack{\gamma^s\\T_\omega^n\gamma^s\subset B_\rho}}T_\omega^n\gamma^s
  \subset B_{\rho+\delta(n)}\setminus B_{\rho-\delta(n)}
  $$
and therefore
$$
\mu^\omega(\partial\mathscr{S}_\Delta^{N-j})
\le\mu^\omega\left(\bigcup_{n=\Delta}^{N-j}T_\omega^{-n}\left(B_{\rho+\delta(n)}\setminus B_{\rho-\delta(n)}\right)\right)
\le\sum_{n=\Delta}^{N-j}\mu^{\theta^n\omega}(B_{\rho+\delta(n)}\setminus B_{\rho-\delta(n)}).
$$
Hence, by assumption~(VI), using  $r=2\delta(n)=\mathcal{O}(n^{-\kappa})$  if 
$\delta$ decays polynomially with power $\kappa$:
\begin{eqnarray*}
\sum_{n=\Delta}^{N-j}\mu^\omega(\partial\mathscr{S}_\Delta^{N-j}) 
&=& \mathcal{O}(1) \sum_{n=\Delta}^{\infty}\frac{n^{-\kappa\xi}}{\rho^{d_1\beta}}\mu(\ball)\\
&=& \mathcal{O}(\rho^{v(\kappa\xi-1) -d_1\beta}\mu(\ball)) 
\end{eqnarray*}
provided  $\Delta\sim\rho^{-v}$ for some positive $v > \frac{\beta-d_0}{\kappa\xi-1}$ which is determined in Section~\ref{estimate.r}  below. 
If we split $\Delta=\Delta'+\Delta''$
then we can estimate $Y$ as follows:
\begin{align*}
Y&=\left|  \int_M \phi \; T_\omega^{-\Delta'}(\ind_{S_k(\Delta')} ) \, d\mu^\omega
- \int_M \phi \, d\mu^\omega \, \int_M \ind_{S_k(\Delta)} \, d\mu^{\theta^\Delta\omega}\right|\\
&%\hspace{5cm}
\le \lambda(\Delta')\|\phi\|_{Lip}\|\ind_{\tilde{\mathscr{S}}_{\Delta''}^{N-j-p'}}\|_{\mathscr{L}^\infty}
+2\mu^\omega(\partial\mathscr{S}_{\Delta''}^{N-j}).
  \end{align*}
Hence
$$
\mu^\omega(\ball \cap T^{-\Delta} S_k) - \mu^\omega(\ball) \, \mu^{\theta^\Delta\omega}(S_k)
\leq \frac{\lambda(\Delta/2)}{\delta \rho} + \mu^\omega(\ball \setminus B_{\rho - \delta \rho})
+\mathcal{O}(\rho^{v(\kappa\xi-1) -d_1\beta}\mu(\ball)).
$$
A similar estimate from below can be done using $\tilde\phi$. Hence 
\begin{equation} \label{R1est}
\mathcal{R}_1
\leq Nc_1\left(\frac{\lambda(\Delta/2)}{\delta\rho}
+\sup_\omega\mu^\omega(B_{\rho +\delta \rho} \setminus B_{\rho -\delta \rho})\right)
+\mathcal{O}(\rho^{v(\kappa\xi-1) -d_1\beta}).
\end{equation}

%%%%%%%%%%%%%%%%%%%%%%%%%%%%%%%%%%%%%%%%%%%
%%%%%%%%%%%%%%%%%%%%%%%%%%%%%%%%%%%%%%%%%%%
%%%%%%%%%%%%%%%%% ESTIMATE OF R2 
\subsection{Estimating the  terms  $\mathcal{R}_2$}\label{estimate.r2}

We will estimate the measure of each of the summands comprising $\mathcal{R}_2$ individually.
We use the product form of the measures
 $\mu^\omega$.
For that purpose fix $j$ and and let $\gamma^u$ be an unstable local leaf through $B$.
Then we put 
$$
\mathscr{C}_j^\omega(B,\gamma^u)=\{\zeta_{\varphi,j}: \zeta_{\varphi,j}\cap B\not=\varnothing,\varphi\in \mathscr{I}_j^\omega\}
$$
 for the cluster of $j$-cylinders that covers the set $B$,
 where the sets $\zeta_{\varphi,k}$ are the images of imbedded $R$-balls in 
 $T_\omega^j\gamma^u$.
Then, using the distortion property~(III),
\begin{eqnarray*}
\mu^\omega_{\gamma^u}(T_\omega^{-j}\ball\cap \ball)
&\le&\sum_{\zeta\in\mathscr{C}_j^\omega(\ball,\gamma^u)}\frac{\mu^\omega_{\gamma^u}(T_\omega^{-j}\ball\cap \zeta)}{\mu^\omega_{\gamma^u}(\zeta)}\mu^\omega_{\gamma^u}(\zeta)\\
&\le&\sum_{\zeta\in \mathscr{C}_j^\omega(\ball,\gamma^u)}\Theta(j)
\frac{\mu^{\theta^{j}\omega}_{T_\omega^j\gamma^u}(\ball\cap T_\omega^j\zeta)}
{\mu^{\theta^{j}\omega}_{T_\omega^j\gamma^u}(T_\omega^j\zeta)}\mu^\omega_{\gamma^u}(\zeta)
\end{eqnarray*}

Since $\mu^{\theta^{j}\omega}_{T_\omega^j\gamma^u}(T_\omega^j\zeta)=\mu^{\theta^{j}\omega}_{T_\omega^j\gamma^u}(B_{R,\gamma^u}(y_k))$
 (for some $y_k$) is uniformly bounded from below, we obtain  
\begin{eqnarray*}
\mu^\omega_{\gamma^u}(T_\omega^{-j}\ball\cap\ball)
&\le& \Theta(j)\mu^{\theta^{j}\omega}_{T_\omega^j\gamma^u}(\ball)
\sum_{\zeta\in \mathscr{C}_j^\omega(\ball,\gamma^u)}\mu^\omega_{\gamma^u}(\zeta)\\
&\le& \Theta(j)\mu^{\theta^{j}\omega}_{T_\omega^j\gamma^u}(\ball)\,  
L\,\mu^\omega_{\gamma^u}\!\left(\bigcup_{\zeta\in \mathscr{C}_j^\omega(\ball,\gamma^u)}\zeta\right)
\end{eqnarray*}
Now, since $\diam\bigcup_{\zeta\in \mathscr{C}_j^\omega(\ball,\gamma^u)}\zeta
\le \delta(j)+\diam \ball\le c_1\delta(j)$ 
(as we can assume that $\rho<\delta(j)$) we obtain
$$
\mu^\omega_{\gamma^u}(T_\omega^{-j}\ball\cap \ball)
\le c_3\Theta(j)\mu^{\theta^{j}\omega}_{T_\omega^j\gamma^u}(\ball)\delta(j)^{u_0}.
$$
Since $d\mu^\omega=d\mu^\omega_{\gamma^u}d\upsilon^\omega(\gamma^u)$
we obtain
$$
\mu^\omega(T_\omega^{-j}\ball\cap \ball)
\le c_4\Theta(j)\mu^{\theta^{j}\omega}(\ball)\delta(j)^{u_0}.
$$
Summing up the $\mu^\omega(T_\omega^{-j}\ball\cap \ball)$  over $j=J,\dots,\Delta-1$, we get
\begin{eqnarray} 
\mathcal{R}'_2(\omega)
&=&\mu^\omega(\ball\cap T_\omega^{-J}\{\tau^{\theta^J\omega}_{\ball}<\Delta-J\})
+\mu^\omega(\ball\cap \{\tau^{\omega}_{\ball}<J\})\notag\\
&\le&\sum_{j=J}^{\Delta-1} \mu^\omega(T_\omega^{-j}\ball\cap \ball)
+\mu^\omega(\ball\cap \{\tau^{\omega}_{\ball}<J\})\notag\\
&\le& c_4\sum_{j=J}^{\Delta-1} \Theta(j)\delta(j)^{u_0}\mu^{\theta^{j}\omega}(\ball)
+\mu^\omega(\ball\cap \{\tau^{\omega}_{\ball}<J\}).
\label{R2 summand ctd3}
\end{eqnarray}
For the entire error term we thus obtain
\begin{eqnarray*}
\mathcal{R}_2
&=&\sum_{k=1}^N\mathcal{R}'_2(\theta^k\omega)\\
&\le& c_4\sum_{j=J}^\Delta\Theta(j)\delta(j)^{u_0}\sum_{k=1}^N\mu^{\theta^{k+j}\omega}(\ball)
+\sum_{k=1}^N\mu^{\theta^k\omega}(\ball\cap \{\tau^{\theta^k\omega}_{\ball}<J\})\\
&\le &c_6t\delta(J)^{u_1}J^{\kappa'}\sum_{k=1}^N\mu^{\theta^{k}\omega}(\ball)
+\mu^\omega(Y^\omega_{\mathsf{x},\rho,t})
\end{eqnarray*}
for some $c_5, c_6$ and almost every $\omega$ and $\rho$ small enough (depending on 
$\omega$). The exponent $u_1$ equals $u_0$ if $\delta(j)$ decays super polynomially
and equals $u_0-\frac1\kappa$ if $\delta(j)$ decays polynomially with power $\kappa$.

%%%%%%%%%%%%%%%%%%%%%%%%%%%%%%%%%%%%%%%%%%%
%%%%%%%%%%%%%%%%%%%%%%%%%%%%%%%%%%%%%%%%%%%
%%%%%%%%%%%%%%%%% ESTIMATE OF R3
\subsection{Estimating the  terms  $\mathcal{R}_3$}\label{estimate.r3}
Assumption~(V) yields 
$$
\mu^\omega(\ball) = \int\mu^\omega_{\gamma^u}(\ball)\,d\upsilon^\omega(\gamma^u) \le\int C_1\rho^{u_0}\,d\upsilon^\omega(\gamma^u)= C_1\rho^{u_0}
$$
for every $\omega$. Since 
$\mu^{\theta^{j}\omega}(\tau_{\ball}^{\theta^j\omega}\le\Delta)\le \sum\limits_{k=1}^{\Delta}\mu^{\theta^{j+k}\omega}(\ball)$
we obtain by Assumption~(V)

\begin{eqnarray*}
\mathcal{R}_3&=&\sum_{j=1}^N\mu^{\theta^j\omega}(\ball)\mu^{\theta^{j}\omega}(\tau_{\ball}^{\theta^j\omega}\le\Delta)\\
&\le& C_1\rho^{u_0}\sum_{j=1}^N\sum\limits_{k=1}^{\Delta}\mu^{\theta^{j+k}\omega}(\ball)\\
&\le & C_1\rho^{u_0} \Delta \sum_{j=1}^N \mu^{\theta^{j}\omega}(\ball) + C_1\rho^{u_0}   \sum_{j=1}^{\Delta} (\Delta-j) \mu^{\theta^{N+j}\omega}(\ball)\\
&\le &  c_7\rho^{u_0} \Delta\sum_{j=1}^N \mu^{\theta^{j}\omega}(\ball)+ c_7 (\Delta\rho^{u_0})^2
\end{eqnarray*}
for some $c_7$ for almost all $\omega$ and $\rho$ small enough
since the first sum converges $\nu$-almost everywhere to $t$.

%%%%%%%%%%%%%%%%%%%%%%%%%%%%%%%%%%%%%%%%%%%
%%%%%%%%%%%%%%%%%%%%%%%%%%%%%%%%%%%%%%%%%%%
%%%%%%%%%%%%%%%%% ESTIMATE OF R
\subsection{The total error}\label{estimate.r} The total error is 
\begin{eqnarray*}
\mathcal{R}&=&\mathcal{R}_1+\mathcal{R}_2+\mathcal{R}_3\\
&\le &Nc_1\left(\frac{\lambda(\Delta)}{\delta\rho}
+\sup_\omega\mu^\omega(B_{\rho +\delta \rho} \setminus B_{\rho -\delta \rho})\right)
+\mathcal{O}(\rho^{v(\kappa\xi-1)-d_1\beta})\\
&&\hspace{2cm}+\left(c_6\delta(J)^{u_1}J^{\kappa'}+C_1\rho^{u_0} \Delta\right)
\sum_{j=1}^N \mu^{\theta^{j}\omega}(\ball)
+c_7 (\Delta\rho^{u_0})^2
+\mu^\omega(Y^\omega_{\mathsf{x},\rho,t}).
\end{eqnarray*}
Let us consider the case when $\lambda$ decays polynomially with power $p$,
i.e.\ $\lambda(k)\sim k^{-p}$. We can choose $\Delta = \rho^{-v}$ so that 
$\lambda(\Delta)=\mathcal{O}(\rho^{-vp})=\mathcal{O}(\rho^w\mu(\ball))\rho^{-(vp-w-d-1)}$  and so that then for some $\epsilon>0$:
$$
\Delta\rho^{u_0}\le c_8\rho^{u_0-\frac{w}p}\mu(\ball)^{-\frac1p}<c_9\rho^{u_0-\frac{w+d_1}p}
=\mathcal{O}(\rho^\epsilon),
$$
that is $ \frac{d_1+w}{p}<v<\min\{1,u_0\}$. 
We then obtain with $\delta\rho=\rho^w$ that 
$N\frac{\lambda(\Delta)}{\delta\rho}\le c_{10}\frac1{\mu(\ball)}\frac{\Delta^{-p}}{\rho^w}
=\mathcal{O}(\rho^{-(vp-w-d-1)})$.
The second term is estimated by (maybe some smaller $\epsilon>0$)
$$
\sup_\omega\mu^\omega(B_{\rho +\delta \rho} \setminus B_{\rho -\delta \rho}) 
= \mathcal{O}(\frac{\rho^{w\xi}}{\rho^\beta}) = \mathcal{O}(\rho^{w\xi-\beta})
=\mathcal{O}(\rho^\epsilon)
$$
since $w\xi>\beta$. Hence we need constants $w,v>0$ such that the following inequalities hold:
\begin{enumerate}
\item $\frac{d_1+w}{p} < v < \min\{1,u_0\};$
\item $v(\kappa\xi-1)-d_1\beta > 0$ from Section~\ref{estimate.r1}; and
\item $w<\frac\xi\beta$ ($w$ can be arbitrarily close to $\frac\xi\beta$).
\end{enumerate}
These conditions hold if we require that 
$\max\{\frac{d_1\beta}{\kappa\xi-1}, \left(\frac{\beta}{\xi}+d_1\right)\frac{1}{p} \} < \min\{1,u_0\}$
which in particular implies $\rho^{v(\kappa\xi-1)-d_1\beta}=\mathcal{O}(\rho^\epsilon)$.
We therefore obtain
$$
\mathcal{R}\le\mathcal{O}(\rho^\epsilon)
+\mathcal{O}(J^{\kappa'}\delta(J)^{u_1}+\rho^\epsilon)\sum_{j=1}^N \mu^{\theta^{j}\omega}(\ball)
+\mu^\omega(Y^\omega_{\mathsf{x},\rho,t})
$$
for all $\rho$ small enough and every $\mathsf{x}$.
\end{proof}

\section{Proof of the main theorems for hitting times}

%%%%%%%%%%%%%%%%%%%%%%%%%%%%%%%%%%%%%%%%
 %%%%%%%%%%%%%%%%%%%%%%%%%%%%%%%%%%%%%%%%
 %%%%%%%     PROOF OF THE MAIN  THEOREM
 %%%%%%%%%%%%%%%%%%%%%%%%%%%%%%%%%%%%%%%%
 \begin{proof}[Proof of Theorem~\ref{main.theorem} and Theorem~\ref{main.theorem.polynomial} for hitting times]
According to~\cite{RSV} Lemma~14 the variance (as a function of $\omega$) of 
$$
\mu^\omega(Z^{\omega}_{\mathsf{x},\rho,t})=\sum_{j=1}^N\mu^{\theta^j\omega}(\ball)
$$
 is bounded by $\rho^q$ for some $0<q<\frac{(p-1)d_0\xi - d_1\xi - \beta}{p\xi+1}$
 and for all $\mathsf{x}$.
 Hence we obtain along any sequence $\rho_i$ for which $\sum_{i=1}^\infty\rho_i^q<\infty$ 
 by an application of Chebycheff's inequality and the  Borel-Cantelli lemma that
$$
\mu^\omega(Z^{\omega}_{\mathsf{x},\rho_i,t})\to t\qquad\mbox{ as } i\to\infty
$$
for $\nu$-almost every $\omega$
since $\mu(Z^{\omega}_{\mathsf{x},\rho,t})=t$ for all $\rho>0$. Since
\begin{eqnarray*}
\prod_{j=1}^N(1-\mu^{\theta^j\omega}(\ball))
&=&\exp\sum_{j=1}^N\left(\mu^{\theta^j\omega}(\ball)+\mathcal{O}(\mu^{\theta^j\omega}(\ball))^2\right)\\
&=&\exp\sum_{j=1}^N\mu^{\theta^j\omega}(\ball)\!\left(1
+\mathcal{O}(\max_j\mu^{\theta^j\omega}(\ball))\right)\\
&\longrightarrow&e^{-t}
\end{eqnarray*}
as $\rho\to0$ along a sequence $\rho_i$, 
since $\max_j\mu^{\theta^j\omega}(\ball)\le C_1\rho^{u_0}\to 0$.

Thus, our theorems are proven along a sequence $\rho_i$ provided we prove that the last
term on the right hand side 
of the inequality in Proposition~\ref{main.proposition} goes to zero as $\rho_i$.
 %%%%%%%%%%%%%%%%%%%%%%%%%%%%%%%%%%%%%%%%
 %%%%%%%%%%%%%%%%%%%%%%%%%%%%%%%%%%%%%%%%
 %%%%%%%      ALMOST SURE CONVERGENCE
That is we must show that for 
almost every $x$ and for almost every $\omega$ the quantity $\mu^\omega(Y^\omega_{\mathsf{x},\rho_i})$
 goes to zero. For this purpose, we use Lemma~\ref{integral.very.short.returns} and let 
 $\gamma''\in(0,\gamma)$ which can be chosen arbitrarily close to $\gamma$.
 Let $\mathcal{B}^\omega_{\rho_i}$ be the set given by Lemma~\ref{integral.very.short.returns}.
 Then $\mu^\omega(\mathcal{B}^\omega_\rho)\lesssim \abs{\log\rho}^{-\gamma''}$
 and if we let  $\alpha\in(\frac{1}{\gamma''},1)$ then we put $\rho_i=e^{-i^\alpha}$ for all  $i\in\mathbb{N}$
 large enough.
These choices imply that $\sum_{i=1}^\infty\rho_i^q<\infty$. Since
   %$r_j=\rho_{j-1}-\rho_j=\rho_j\mathcal{O}(j^{-(1-\alpha)})$
   %and 
   $$
   \mu^\omega(\mathcal{B}^\omega_{\rho_i})\le c_1i^{-\alpha\gamma''}
   $$
we thus obtain $\sum_i \mu^\omega(\mathcal{B}^\omega_{\rho_i})<\infty$ (as $\alpha\gamma''>1$). 
By the Borel-Cantelli lemma we now conclude that
    $\mu^\omega(x\in \mathcal{B}^\omega_{\rho_i} \mbox{ i.o.})=0$ which implies that
 $\mu^\omega(Y^\omega_{\mathsf{x},\rho_i,t})\lesssim \abs{\log\rho_i}^{-(\gamma-\gamma'')}\longrightarrow0$
 as $i\to\infty$ for almost every $\mathsf{x}$ and $\omega$.

    This concludes the proof of  the two main theorems for hitting times along the sequence  $\rho_i$.

    In order to get the convergence for arbitrary $\rho\to0$ 
   let $\rho>0$ be sufficiently small and $i$ so that $\rho_i\le\rho\le\rho_{i-1}$. We have $r_i=\rho_{i-1}-\rho_i=\rho_i.\mathcal{O}(i^{-(1-\alpha)})$,  then
\begin{eqnarray*}
& & \left|\mu^\omega\!\left(\tau^\omega_{\ball}>\frac{t}{\mu(\ball)}\right)
-\mu^\omega\!\left(\tau^\omega_{B_{\rho_i}}>\frac{t}{\mu(B_{\rho_i})}\right)\right|\\
&\leq& \mu^\omega\!\left(\tau^\omega_{\ball\setminus B_{\rho_i}}<\frac{t}{\mu(\ball)}\right)
+\mu^\omega\!\left(\frac{t}{\mu(B_{\rho})}<\tau^\omega_{B_{\rho_i}}<\frac{t}{\mu(B_{\rho_i})}\right)\\
&\le &\sum_{j=1}^{\frac{t}{\mu(B_\rho)}}\mu^{\theta^j\omega}(\ball\setminus B_{\rho_i})
+\sum_{j=\frac{t}{\mu(\ball)}}^{\frac{t}{\mu(B_{\rho_i})}}\mu^{\theta^j\omega}(B_{\rho_i})\\
&\leq&{\mu(B_\rho)}\sum_{j=1}^{\frac{t}{\mu(B_\rho)}}\frac{\mu^{\theta^j\omega}(B_{\rho+r_i}\setminus B_{\rho-r_i})}{\mu(B_\rho)}
+K\mu(B_{\rho_i})\left|\frac{t}{\mu(\ball)}-\frac{t}{\mu(B_{\rho_i})}\right|\\
&\lesssim&t\frac{r_i^\xi}{\rho_i^\beta}
+Kt\frac{\mu(\ball\setminus B_{\rho_i})}{\mu(B_{\rho})}\\
&\lesssim&t(1+K)\frac{r_i^\xi}{\rho_i^\beta}\\
&\lesssim&\frac{\rho_i^{\xi-\beta}}{i^{\xi(1-\alpha)}}
\end{eqnarray*}
using Assumption~(VI) and Assumption~(V). This difference goes to zero as $i\to\infty$ since
$\xi\ge\beta$ and $1-\alpha>0$ which concludes the proof of the theorem.
\end{proof}

%%%%%%%%%%%%%%%%%%%%%%%%%%%%%%%%%%%%%%%%%%%%%%%%%
%%%%%%%%%%%%%%%%%%%%%%%%%%%%%%%%%%%%%%%%%%%%%%%%%
%%%%%%%%     SECTION: RETUN TIMES DISTRIBUTION`
%%%%%%%%%%%%%%%%%%%%%%%%%%%%%%%%%%%%%%%%%%%%%%%%%
\section{Return times distribution}\label{return.times.distribution}
\begin{proof}[Proof of Theorem~\ref{main.theorem1} and Theorem~\ref{main.theorem.polynomial} for return times]
Since we proved an exponential distribution for the hitting times, to get an exponential distribution for the return times we will estimate the difference between the hitting time statistics and the return time statistics. To do so, we use the definition of $\epsilon_\omega$ in \eqref{defeps} to notice that
$$
\left|\mu^\omega_{\ball}(\tau^\omega_{\ball}>N)-\mu^\omega(\tau^\omega_{\ball}>N)\right|\leq\frac{\epsilon_\omega(\ball)}{\mu^\omega(\ball)},
$$
%where we split the second term as follows:
%$$
%\mu^\omega(\ball\cap\{\tau^\omega_{\ball}>N\})
%=\mu^\omega(\ball)\mu^{\theta\omega}(\tau^{\theta\omega}_{\ball}>N-1)
%+\epsilon_\omega(\ball),
%$$
%where $\epsilon_\omega$ is as in the proof of Proposition~\ref{main.proposition}.
Observe that the convergence of the RTS does not come immediately from the first part of the theorem, as one could have hoped for  from the deterministic case (e.g. \cite{Sau09}) or the annealed case \cite{Rou14}.

 In order to estimate the error term we split the RHS in three terms as in~\eqref{estsumeps}:
 $$
 \left|\frac{\epsilon_\omega(\ball)}{\mu^\omega(\ball)}\right|
 =\tilde{\mathcal{R}}\leq \tilde{\mathcal{R}}_1+\tilde{\mathcal{R}}_2+\tilde{\mathcal{R}}_3.
 $$
 We estimate the first term $\tilde{\mathcal{R}}_1$, the decay of correlations term
 (unlike the term $\mathcal{R}_1$ in Theorem~\ref{main.theorem} there is no factor  $N$):
 $$
\tilde{\mathcal{R}}_1 =\frac{1}{\mu^\omega(\ball)} \sup_{k\ge1}\left| \mu^\omega(\ball \cap T_\omega^{-\Delta}S_k) 
- \mu^\omega(\ball) \, \mu^{\theta^{\Delta}\omega}(S_k) \right|.
$$
As in Section~\ref{estimate.r1}, we get:
$$
\tilde{\mathcal{R}}_1
\leq \frac{c_1}{\mu^\omega(\ball)}\left(\frac{\lambda(\Delta/2)}{\delta\rho}
+\mu^\omega(B_{\rho +\delta \rho} \setminus B_{\rho -\delta \rho})\right)
+\frac{\mu(\ball)}{\mu^\omega(\ball)}\mathcal{O}(\rho^{v(\kappa\xi-1)-d_1\beta})
$$
where we use Assumption~(V) to estimate the term $\frac{\mu(\ball)}{\mu^\omega(\ball)}$
by $K$.
The short hitting times term, $\tilde{\mathcal{R}}_3$, can be dealt with easily as in Section \ref{estimate.r3}:
$$
\tilde{\mathcal{R}}_3=\mu^\omega( \{y:\tau^\omega_{\ball}(y) < \Delta\}) 
\leq\sum_{k=1}^{\Delta}\mu^{\theta^k\omega}(\ball)
\leq c_2\Delta\rho^{u_0}.
$$
We are left with the short return times term, $\tilde{\mathcal{R}}_2$ which we
estimate similarly to Section~\ref{estimate.r2}:
$$
\tilde{\mathcal{R}}_2
=\frac{\mathcal{R}'_2(\omega)}{\mu^\omega(\ball)}
= \frac{1}{\mu^\omega(\ball)}\mu^\omega(\ball \cap \{y:\tau^\omega_{\ball}(y) < \Delta\})
\le c_3\sum_{j=J}^\Delta\Theta(j)\delta(j)^{u_0}\frac{\mu^{\theta^j\omega}(\ball)}{\mu^\omega(\ball)}
$$
for all $\mathsf{x}\not\in\mathcal{V}^\omega_{4\rho}$
 (as $\ball(\mathsf{x})\cap\{\tau^\omega_{\ball(\mathsf{x})}<J\}=\varnothing$ 
 for such $\mathsf{x}$).
This implies by Assumption~(V)
$\tilde{\mathcal{R}}_2\le c_4 \delta(J)^{u_1}J^{\kappa'}$.

Consequently, proceeding as in Section~\ref{estimate.r} we finally obtain
$$
\tilde{\mathcal{R}}\le c_5\!\left(\delta(J)^{u_0}J^{\kappa'}+\rho^\epsilon\right)
$$
for some $\epsilon>$ and for all $\mathsf{x}\not\in\mathcal{V}^\omega_{4\rho}$.
\end{proof}

%%%%%%%%%%%%%%%%%%%%%%%%%%%%%%%%%%%%%%%%
 %%%%%%%%%%%%%%%%%%%%%%%%%%%%%%%%%%%%%%%%
 %%%%%%%     SECTION: ALTERNATIVE PROOF
 %%%%%%%%%%%%%%%%%%%%%%%%%%%%%%%%%%%%%%%%
 \subsection{Alternative argument for $\tilde{\mathcal{R}}_2\to0$ almost surely}\label{section.alternative}
Here, we will give an alternative solution to prove that $\tilde{\mathcal{R}}_2\to0$ almost surely.

First, we will estimate the measure of the set 
$\mathcal{W}^\omega_\rho= \{y: d(T_\omega^jy,y)<2\rho \text{ for some } 1< j<\Delta\}$.
%Then $\{\tau^\omega_{\ball}<\Delta\}\subset\bigcup_{j=1}^\Delta\{y: d(y,T_\omega^j y)<2\rho\}$
%and consequently
%$$
%\mu^\omega(\ball\cap\{\tau^\omega_{\ball}<\Delta\})
%\le\mu^\omega(\ball\cap\mathcal{W}^\omega_\rho).
%$$
Let $J = \mathfrak{ab}|\log \rho|$ where $\mathfrak{a}, \mathfrak{b}$ be as before and put 
\begin{eqnarray*}
\mathcal{W}^{\omega,1}_\rho &=&\{y:d(T_\omega^jy,y)<2\rho \text{ for some } 0< j\le J\},\\
\mathcal{W}^{\omega,2}_\rho &=&\{y:d(T_\omega^jy,y)<2\rho \text{ for some } J< j<\Delta\}.
\end{eqnarray*}
We will get  estimates on the measure of $\mathcal{W}^{\omega,i}$, $i=1,2$.

 %%%%%%%%%%%%%%%%%%%%%%%%%%%%%%%%%%%%%%%%
 %%%%%%%     FIRST ESTIMATE
 %%%%%%%%%%%%%%%%%%%%%%%%%%%%%%%%%%%%%%%%
\subsection{Estimating $\mathcal{W}_\rho^{\omega,2}$}
We follow the proof of Proposition~\ref{prop.short.returns} and put $\mathcal{M}^\omega_\rho(n)= \{y:d(T_\omega^ny,y) < 2\rho\}$ the level set. As in~\eqref{level_summand} one has now
$$
\mu^\omega(\mathcal{M}^\omega_\rho(n)) 
\le \sum_{\zeta}\Theta(n)  \,\mu^\omega_{\hat\gamma^u}(T_\omega^{n}(T_\omega^{-n}\mathcal{M}^\omega_{\rho}(n) \cap \zeta)) \,\mu_{\gamma^u}^{\tilde{\omega}}(\zeta).
$$
To get a bound on the diameter of the set 
$T_\omega^{n}(T_\omega^{-n}\mathcal{M}^\omega_{\rho}(n) \cap \zeta)$ we take points 
$x,y \in T_\omega^{-n}\mathcal{M}^\omega_{\rho}(n) \cap \zeta$. 
Then
$$
d(T_\omega^nx, T_\omega^ny) \le d(T_\omega^nx, x)+d(x, y)+d(y, T_\omega^ny) \le 4\rho+ \delta(n).
$$
As a result,
$$
\mu^\omega(\mathcal{M}^\omega_\rho(n)) 
\le c_1\Theta(n)\left( \rho^{u_0} + \delta(n)^{u_0}\right)\sum_{\zeta:\,\zeta\cap T_\omega^{-n}\mathcal{M}^\omega_\rho\not=\varnothing}\mu^{\tilde\omega}_{\gamma^u}(\zeta).
$$
Since the sum of $\zeta$ is bounded, we thus get with $\Theta(n)=\mathcal{O}(n^{\kappa'})$
$$
\mu^\omega(\mathcal{W}_\rho^{\omega,2}) 
\le \sum_{n=J}^{\Delta}\mu^\omega(\mathcal{M}^\omega_\rho(n))
 \le c_2 \left(\Delta^{\kappa'+1}\rho^{u_0} + \delta(J)^{u_1}J^{\kappa'}\right)
$$
since $J = \mathfrak{ab}|\log \rho|$ where $u_1=u_0$ if $\delta$ is superpolynomial
and $u_1=u_0-\frac1\kappa$ if $\delta$ decays polynomially at rate $\kappa$. 

 %%%%%%%%%%%%%%%%%%%%%%%%%%%%%%%%%%%%%%%%
 %%%%%%%     SECOND ESTIMATE
 %%%%%%%%%%%%%%%%%%%%%%%%%%%%%%%%%%%%%%%%
\subsubsection{Estimating $\mathcal{W}_\rho^{\omega,1}$}
Notice that $\mathcal{W}_\rho^{\omega,1} \subset \mathcal{V}_{2\rho}^{\omega}$ 
and by Proposition~\ref{prop.short.returns}
$$
\mu^\omega(\mathcal{W}_\rho^{\omega,1} )\le \mu^\omega(\mathcal{V}_{2\rho}^{\omega}) 
 \le C_3\!\left(e^{-\mathfrak{c}\abs{\log 2\rho}^{1/2}}  
 +\delta(\mathfrak{ab}\abs{\log 2\rho})^{u_1}|\log2\rho|^{\kappa'}\right).
$$
 %%%%%%%%%%%%%%%%%%%%%%%%%%%%%%%%%%%%%%%%
 %%%%%%%     SECTION: ALTERNATIVE PROOF
 %%%%%%%%%%%%%%%%%%%%%%%%%%%%%%%%%%%%%%%%
\subsubsection{Almost sure limit of $\tilde{\mathcal{R}}_2$}
To estimate $\mathcal{W}^\omega_\rho=\mathcal{W}_\rho^{\omega,1}\cup\mathcal{W}_\rho^{\omega,2}$
we choose $\Delta=\rho^{-v'}$ and obtain
$$
\mu^\omega(\{y: \tau^\omega_{2\rho}(y)<\Delta\})=\mu^\omega(y:d(T_\omega^jy,y)<2\rho \text{ for some } 0<j<\Delta)= \mathcal{O}(\abs{\log \rho}^{-b}),
$$
where  $b=u_1\kappa-\kappa'$ ($b>0$) assuming $\delta$ decays polynomially
with power $\kappa$ and $0<v'<\frac{u_0}{\kappa'+1}$.

In order to get a limit for $\rho\to0$,  we apply the Borel-Cantelli Lemma to the 
sequence $\rho_n=e^{-n^{2/b}}$ and obtain that for $\mu^\omega-$almost every $y$
and all $n$ large enough:
$\tau^\omega_{2\rho_n}(y)\geq \Delta_n=\rho_n^{-v'}$ or
$$
\tau^\omega_{2e^{-n^{2/b}}}(y)\geq e^{v' n^{2/b}}
$$
which implies 
$$
\liminf _{n\rightarrow\infty}\frac{\log\tau^\omega_{2e^{-n^{2/b}}}(y)}{-\log2e^{-n^{2/b}}}
\geq \liminf_{n\rightarrow\infty} \frac{\log e^{v' n^{2/b}}}{-\log2e^{-n^{2/b}}}=v'.
$$
For every $\rho>0$ small enough, there is an $n$ so that $\rho_{n+1}\le \rho<\rho_n$ 
and consequently
$$
\frac{\log\tau^\omega_{2\rho_{n}}(y)}{-\log\rho_{n+1}}
\le\frac{\log\tau^\omega_{2\rho}(y)}{-\log\rho}
\le\frac{\log\tau^\omega_{2\rho_{n+1}}(y)}{-\log\rho_n}.
$$
As $\frac{\log\rho_{n+1}}{\log\rho_n}=\frac{(n+1)^2}{n^2}\to 1$ as $n\to\infty$, we conclude
that
$$
\liminf _{\rho\rightarrow0}\frac{\log\tau^\omega_{2\rho}(y)}{-\log\rho}
=\liminf _{n\rightarrow\infty}\frac{\log\tau^\omega_{2e^{-n^{2/b}}}(y)}{-\log2e^{-n^{2/b}}}
\geq v'.
$$
In other words for any $0<v<v'$ one has $\mu^\omega(\mathcal{L}(\rho_0))\rightarrow1$ 
as $\rho_0\rightarrow0$ where
$$
\mathcal{L}(\rho_0)=\{y: \tau^\omega_{2\rho}(y)>\rho^{-v}\;\forall \rho<\rho_0\}.
$$
Finally, this implies by the Lebesgue density theorem (following the proof of Lemma~42 of~\cite{Sau09}) that
$$
\tilde{\mathcal{R}}_2 = \frac{1}{\mu^\omega(\ball)}\mu^\omega(\ball \cap \{y:\tau^\omega_{\ball}(y) < \Delta\})\underset{\rho \rightarrow0}{\longrightarrow}0
$$
for $\mu^\omega$-almost every $x$, where here $\Delta=\rho^{-v}$, $v<\frac{u_0}{\kappa'+1}$.

%%%%%%%%%%%%%%%%%%%%%%%%%%%%%%%%%%%%%%%%%%%%%%%%%
%%%%%%%%%%%%%%%%%%%%%%%%%%%%%%%%%%%%%%%%%%%%%%%%%
%%%%%%%%%%%%         SECTION: EXAMPLES
%%%%%%%%%%%%%%%%%%%%%%%%%%%%%%%%%%%%%%%%%%%%%%%%%
\section{Examples}\label{example}

%%%%%%%%%%%%%%%%%%%%%%%%%%%%%%%%%%%%%%%%%%%%%%%%%
%%%%%%%%%%%%%%%%%%%%%%%%%%%%%%%%%%%%%%%%%%%%%%%%%
%%%%%%%%%%%%         SUBSECTION: EXPANDING INTERVAL MAPS
%%%%%%%%%%%%%%%%%%%%%%%%%%%%%%%%%%%%%%%%%%%%%%%%%
\subsection{Random $C^2$ interval maps}
As an example we consider random maps on the unit interval $I$.
As above let  $S:\Omega\times I\circlearrowleft$
be a skew action where the map $\theta$ is acting invertibly on $\Omega$.
For each $\omega$ the map $T_\omega:I\to I$ is a piecewise expanding
map on the interval $I$. We assume that $T_\omega$ is piecewise $C^2$ 
with uniformly bounded $C^2$ norms. For $\varphi\in\mathscr{I}^\omega_n$ we denote by 
$\zeta_\varphi=\varphi(I)$ the $n$-cylinder associated with $\varphi$. 
As before, put 
$$\delta(n)=\sup_\omega\max_{\varphi\in\mathscr{I}^\omega_n}|\zeta_\varphi|.
$$
 For a function $\psi:I\to\mathbb{R}$ be denote by $\var\,\psi$ its variation on the
 unit interval and let 
$$
\|f\|=\var\,\psi +\|f\|_{\mathscr{L}^1}
$$ 
be its norm. This makes $X=\{f\in C(I,\mathbb{R}), \|f\|<\infty\}$
a Banach space which is equipped with the strong norm $\|\cdot\|$ and 
 the weak norm $\|\cdot\|_{\mathscr{L}^1}$.
Consider the transfer operator $\mathcal{L}$ on $X$ which for each $\omega$ 
maps a function $\psi\in X$ on the interval to a function $\mathcal{L}_\omega\psi$ on the interval.
It is given by 
$$
\mathcal{L}_\omega \psi(x)=\sum_{\varphi\in\mathscr{I}^\omega_1}
\frac{\psi(\varphi x)}{|DT_\omega(\varphi x)|}.
$$
The iterates of the transfer operator are 
$\mathcal{L}_\omega^n
=\mathcal{L}_{\theta^{n-1}\omega}\circ\cdots\circ\mathcal{L}_{\theta\omega}\circ\mathcal{L}_\omega$.
We shall next prove the Doeblin-Fortet inequality:

%%%%%%%%%%%%%%%%%%%%%%%%%%%%%%%%%%%%%%%%%%%%%%%%%
%%%%%%%%%%%%%%%%%%%%%%%%%%%%%%%%%%%%%%%%%%%%%%%%%
%%%%%%%%%%%     LEMMA: DOEBLING FORTET
\begin{lem}\label{DoeblinF}
Assume that $\delta(k)$ decreases to zero as $k\to\infty$.
Then there exist $\eta<1$, $n\in\mathbb{N}$ and a constant $C_4$ so that
for every $\omega\in\Omega$ and $\psi:I\to\mathbb{R}$ with $\|\psi\|<\infty$ one has
$$
\var\,\mathcal{L}_\omega^n\psi\le\eta\,\var \,\psi+C_4 \|\psi\|_{\mathscr{L}^1}.
$$
\end{lem}

\begin{proof} Let us fix $\omega$.
In order to estimate $\var\,|DT_\omega^n|^{-1}$ let $\ell\in\mathbb{N}$ and $n=p\ell$. Then, for $\varphi\in \mathscr{I}^\omega_n$, $|DT_\omega^{p\ell}|^{-1}\varphi
=(|DT_\omega^{(p-1)\ell}|^{-1}T_\omega^\ell\varphi)|DT_\omega^\ell|^{-1}\varphi$ 
and 
$$
\var\,|DT_\omega^{p\ell}|^{-1}\varphi
\le||DT_\omega^{(p-1)\ell}|^{-1}T_\omega^\ell\varphi|_\infty\var\, |DT_\omega^\ell|^{-1}\varphi
+||DT_\omega^\ell|^{-1}\varphi|_\infty\var\,|DT_\omega^{(p-1)\ell}|^{-1}T_\omega^\ell\varphi.
$$
There exists a constant $c_1$ so that 
$$
||DT_\omega^{\ell}|^{-1}\varphi|_\infty\le c_1|\zeta_\varphi|\le c_1\delta(\ell)
$$
and similarly 
$|DT_\omega^{(p-1)\ell}|^{-1}\varphi|_\infty\le c_1\delta((p-1)\ell)$. 
Hence since $T_\omega^\ell\varphi\in\mathscr{I}^{\theta^\ell\omega}_{(p-1)\ell}$
one has 
$$
\var\,|DT_\omega^{p\ell}|^{-1}\varphi\le c_1\delta(\ell)\var\,|DT_\omega^{(p-1)\ell}|^{-1}T_\omega^\ell\varphi+c_1\delta((p-1)\ell)\var\,|DT_\omega^{\ell}|^{-1}\varphi.
$$
Recursively one obtains for all $\varphi\in\mathscr{I}^\omega_{p\ell}$:
\begin{eqnarray*}
\var\,|DT_\omega^{p\ell}|^{-1}\varphi
&\le&\sum_{j=0}^{p-1}(c_1\delta(\ell))^jc_1\delta((p-j-1)\ell)\var\,|DT_\omega^{\ell}|^{-1}T_\omega^{j\ell}\varphi\\
&\le&\sum_{j=0}^{p-1}(c_1\delta(\ell))^jc_1\delta((p-j-1)\ell)\var\,|DT_\omega^{\ell}|^{-1}.
\end{eqnarray*}
If we choose $\ell$ so that $\tilde\eta=c_1\delta(\ell)<1$ then we obtain 
$$
\var\,|DT_\omega^{n}|^{-1}\varphi \le \Delta(n)\var\,|DT_\omega^{\ell}|^{-1}
$$
for all $\varphi\in \mathscr{I}^\omega_n$, for all $n\in\mathbb{N}$ and for some function $\Delta(n)$ which decays to zero at the same rate as 
$\delta(n)$.

The variation of the transfer operator is then estimated as:
$$
\var_I\,\mathcal{L}_\omega^n\psi
=\sum_{\varphi\in\mathscr{I}^\omega_n}(|DT_\omega^{n}|^{-1}\psi)\varphi
\le\sum_{\varphi\in\mathscr{I}^\omega_n}
\left(\var\, \psi\varphi||DT_\omega^{n}|^{-1}\varphi|_\infty+|\psi\varphi|_\infty\var\,|DT_\omega^{n}|^{-1}\varphi\right).
$$
Since $\left||DT_\omega^{n}|^{-1}\right|_\infty\le c_1|\zeta_{\varphi}|\le c_1\delta(n)$ one obtains
$$
\var_I\,\mathcal{L}_\omega^n\psi
\le c_1\delta(n)\sum_{\varphi}\var\, \psi\varphi
+c_2\Delta(n)\sum_\varphi |\psi\varphi|_\infty
$$
where $c_2=\sup_\omega|DT_\omega^\ell|^{-1}$.
With the estimate 
$|\psi\varphi|_\infty\le \int_I|\psi\varphi|\,d\lambda+\var\,\psi\varphi$
this leads to
$$
\var_I\,\mathcal{L}_\omega^n\psi
\le (c_1\delta(n)+c_2\Delta(n))\var\, \psi
+c_2\Delta(n)\left|\frac1{|DT_\omega^{n}|^{-1}}\right|_\infty\int_I\sum_\varphi |\psi\varphi|_\infty |DT_\omega^{n}|^{-1}\varphi\,d\lambda
$$
as $\sum_{\varphi}\var\,\psi\varphi=\var\,\psi$. Since the Lebesgue measure $\lambda$ is 
a fixed point of the transfer operator we finally get
$$
\var_I\,\mathcal{L}_\omega^n\psi
\le (c_1\delta(n)+c_2\Delta(n))\var\, \psi
+c_2\Delta(n)|DT_\omega^{n}|_\infty\int_I|\psi|_\infty \,d\lambda.
$$
Now, if we choose $n$ so that $\eta=c_1\delta(n)+c_2\Delta(n)<1$  we obtain
$$
\var_I\,\mathcal{L}_\omega^n\psi
\le \eta\,\var\, \psi+c_3\|\psi\|_{\mathscr{L}^1},
$$
where $c_3\le c_2\Delta(n)|DT_\omega^{n}|_\infty$.  Put $C_4=c_3$.
Note that the constant $\eta<1$ can be chosen arbitrarily small.
\end{proof}

\noindent This proves the property~(LY2) of~\cite{Bu99}. The other 
two properties~(LY0) and~(LY1) are naturally satisfied as are the 
properties~(V). To verify condition~(RC) let 
$\psi\in\mathcal{C}_a$ where $\mathcal{C}_a=\{\psi>0:\var\,\psi\le a\|\psi\|_{\mathscr{L}^1}\}$.
Then by iterating Lemma~\ref{DoeblinF} one obtains
$$
\var\,\mathcal{L}_\omega^m\psi
\le\eta^\frac{m}n\var\,\psi+\frac{C_4}{1-\eta^\frac1n}\|\psi\|_{\mathscr{L}^1}
\le\left(\eta^\frac{m}n+\frac{C_4}{1-\eta^\frac1n}\right)\|\psi\|_{\mathscr{L}^1}
$$
for all large $m$.
If we choose $a\ge\frac{2C_4}{(1-\eta^\frac1n)^2}$ this implies
 $\inf \mathcal{L}_\omega^m\psi\ge\|\mathcal{L}_\omega^m\psi\|_{\mathscr{L}^1}
-\var\,\mathcal{L}_\omega^m\psi\ge\frac{a}2$.
Hence the condition~(RC)
of~\cite{Bu99} is satisfied with $\alpha_n=\frac{a}2$.

Therefore by the Main Theorem of~\cite{Bu99}
there exists a family of absolutely continuous measure $\mu^\omega$
on the fibres $\{\omega\}\times I$ which satisfy the generalised 
invariance property $T_\omega^*\mu^\omega=\mu^{\theta\omega}$.
In particular there is an $S$-invariant measure $\mathbb{P}$ on $\Omega\times I$
which is of the form $d\mathbb{P}(\omega,x)=d\mu^\omega(x)d\nu(\omega)$,
where $\nu$ is a $\theta$-invariant measure on $\Omega$.
Also note as a consequence of the lower bound $\inf\psi\ge\frac{a}2\|\psi\|_{\mathscr{L}^1}$
the  densities $h_\omega$ of $\mu^\omega$ have a uniform lower bound,
that is there exists a constant $c_1>0$ so that $\inf h_\omega\ge c_1$ for all
$\omega\in\Omega$.

Moreover one has decay of the annealed correlation function~(I)
and also the decay of the quenched correlation functions~(II).
In fact, the decay function $\lambda(n)$ decays exponentially fast to zero.

Since the measures $\mu^\omega$ are absolutely continuous with respect
to Lebesgue measure, condition~(V) is satisfied with any values $d_0<1<d_1$ 
arbitrarily close to $1$ and from the uniform lower bound on the densities $h_\omega$,
i.e.\ we can take $K=1/c_1$. %Similarly condition~(IV) is satisfied with any $u_0<1$.
Condition~(III) follows from the uniform boundedness of second order derivatives.
The annulus condition~(VI) is satisfied with $\xi=\beta=1$.
%Condition~(VIII) follows 
We can therefore invoke Theorems~\ref{main.theorem} and \ref{main.theorem1}
and obtain the following result:

%%%%%%%%%%%%%%%%%%%%%%%%%%%%%%%%%%%%%%%%%%%%%%%%%
%%%%%%%%%%%%        THEOREM: TWICE DIFFERENTIABLE INTERVAL MAPS
%%%%%%%%%%%%%%%%%%%%%%%%%%%%%%%%%%%%%%%%%%%%%%%%%
\begin{thrm} Let $S:\Omega\times I\circlearrowleft$ be a skew system as described above,
where the maps $T_\omega$ are piecewise $C^2$ with uniformly bounded
$C^2$ derivatives. Let $\delta(n)$ be a summable sequence which monotonically 
decreases to zero so that $|\zeta_\varphi|\le\delta(n)$ for all $\varphi\in\mathscr{I}^\omega_n$
for all $n$. Then
$$
\mu^\omega\!\left(y\in[0,1]: 
\tau^\omega_{B_{\rho}(\mathsf{x})}(y)>\frac{t}{\mu(B_{\rho}(\mathsf{x}))}\right)
\longrightarrow e^{-t}\qquad \mbox{ as }\rho\to0
$$
and
$$
\mu^\omega_{B_{\rho}(\mathsf{x})}\!\left(y\in [0,1]: 
\tau^\omega_{B_{\rho}(\mathsf{x})}(y)>\frac{t}{\mu(B_{\rho}(\mathsf{x}))}\right)
\longrightarrow e^{-t}\qquad \mbox{ as }\rho\to0
$$
for all $t>0$ for Lebesgue almost every $\mathsf{x}\in [0,1]$ and
 $\nu$-almost every $\omega\in\Omega$.
\end{thrm}

\noindent Clearly, if the maps $T_\omega$ are uniformly expanding then 
$\delta$ decays exponentially and satisfies the requirement of the theorem.

%%%%%%%%%%%%%%%%%%%%%%%%%%%%%%%%%%%%%%%%%%%%%%%%%
%%%%%%%%%%%%%%%%%%%%%%%%%%%%%%%%%%%%%%%%%%%%%%%%%
%%%%%%%%%%%%         SUBSECTION: PARABOLIC INTERVAL MAPS
%%%%%%%%%%%%%%%%%%%%%%%%%%%%%%%%%%%%%%%%%%%%%%%%%
\subsection{Random parabolic interval maps}

We use the family of Pomeau-Manneville maps indexed by $\alpha>0$ which is given by
$$
T_\alpha(x)=\begin{cases}x+2^{\alpha}x^{1+\alpha}&\mbox{ if $x\in[0,\frac12)$}\\
2x-1&\mbox{ if $x\in[\frac12,1]$}\end{cases}.
$$
These maps have a neutral (parabolic) fixed point at $x=0$ and are otherwise
expanding.
It is known that if $\alpha<1$ then there exists an invariant absolutely continuous 
probability measure. Here we assume the setting of~\cite{BB16}. Let $\Omega=\{0,1\}^{\mathbb{Z}}$ be the `driving space' with the 
left shift map $\theta:\Omega\circlearrowleft$. We equip $\Omega$ with 
the Bernoulli measure $\nu$ with weights $(\frac12, \frac12)$. Let $0<\alpha_0<\alpha_1<1$ 
and define the function $\alpha:\Omega\to\{\alpha_0,\alpha_1\}$ by
$$
\alpha(\omega)=\begin{cases}\alpha_0&\mbox{ if $\omega_0=0$}\\
\alpha_1&\mbox{ if $\omega_0=1$}\end{cases}.
$$
Then we have a skew action  $S:\Omega\times I$, with $I=[0,1]$
defined by $S(\omega,x)=(\theta(\omega),T_{\omega}x)$, 
where we wrote $T_\omega=T_{\alpha(\omega)}$.
Its iterates are $S(\omega,x)=(\theta^n\omega,T_\omega^nx)$, where 
$T_\omega^n=T_{\theta^{n-1}\omega}\circ\cdots\circ T_{\theta\omega}\circ T_\omega$.
It is shown in~\cite{BB16} that there exists an $S$-invariant probability measure
$\mu=h\nu\times\lambda$, where $\lambda$ is the Lebesgue measure on $I$
and where the density $h:\Omega\times I\to\mathbb{R}^+$ is Lipschitz continuous
on compact subsets of $\Omega\times(0,1]$. Notice that we here identify the 
shift space $(\Omega,\theta)$ with the doubling interval map $p_1=p_2=\frac12$ as
in~\cite{BB16}. Let us note that in~\cite{BB16} Lemma~3.1 we can use 
the cone of functions 
$$
\mathcal{C}_2=\left\{f\in C^0((0,1])\cap\mathscr{L}^1(\lambda): f\ge0, f\mbox{ decreasing},
 x^{1+\alpha}f\mbox{ increasing}, f(x)\le ax^{-\alpha}\lambda(f)\right\}
 $$
 where $a$ according to~\cite{AHNTV14} Lemma~1.2  is chosen large enough so that 
 $\mathcal{C}_2$ is invariant under the transfer operators for $T_{\alpha_0}$ and 
 $T_{\alpha_1}$. Then lets us replace the cone $\mathcal{C}_a$ in~\cite{BB16} 
 by $\mathcal{C}_2$ to obtain the invariant density $h$ for the annealed measure $\mu$.
On the fibres $I_\omega=\{\omega\}\times I$ we then have the density $h_\omega$
given by $h_\omega(x)=h(\omega,x)$. This defines the fibred measures
 $\mu^\omega=h_\omega\lambda$ on $I_\omega$ which have the invariance
 property $T_\omega^*\mu^\omega=\mu^{\theta\omega}$.
For the transfer operator $\mathcal{P}_\omega$ (adjoint to $T_\omega$), one has $\mathcal{P}_\omega h_\omega=h_{\theta\omega}$ and 
 $\mathcal{P}_\omega^*\lambda=\lambda$ and also  by~\cite{AHNTV14} Theorem~1.6:
\begin{eqnarray*}
& &\left| \int\psi(\phi\circ T_\omega^n)\,d\mu^\omega
 -\int\psi\,d\mu^\omega\int\phi\,d\mu^{\theta^n\omega}\right|\\
 &=&\left|\int(\mathcal{P}_\omega^n\psi h_\omega)\phi\,d\lambda
 -\mu^\omega(\psi)\int\phi\mathcal{P}_\omega^nh_\omega\,d\lambda\right|\\
 &\le&\int|\phi|\cdot\left|\mathcal{P}_\omega^n(\psi h_\omega
 -h_\omega\mu^\omega(\psi))\right|\,d\lambda\\
&\le & c_1|\phi|_\infty\left(\|\psi h_\omega\|_{\mathscr{L}^1(\lambda)}
 +\|h_\omega\mu^\omega(\psi)\|_{\mathscr{L}^1(\lambda)}\right)\frac{\log^\frac1{\alpha_1}n}{n^{\frac1{\alpha_1}-1}},
 \end{eqnarray*}
 for some constant $c_1$ (which by~\cite{AHNTV14} is equal to $\max\{C_{\alpha_0},C_{\alpha_1}\}$).
 This is under the stated assumption that $\phi h_\omega$ and $h_\omega\mu^\omega(\psi)$ 
 belong to the cone of functions $\mathcal{C}_2$. This in fact applies to the function 
 $h_\omega\mu^\omega(\psi)$. A careful reading of the proof makes it apparent 
 that the class of functions to which the contraction applies is far wider and in fact is
 only determined by the property that $\|\phi-\mathbb{A}_\epsilon\phi\|_{\mathscr{L}^1}$
 is bounded by a multiple of $\epsilon^{1-\alpha}$. The smoothing operator $\mathbb{A}_\epsilon$
 is given by $\mathbb{A}_\epsilon\phi(x)=\frac1{2\epsilon}\int_{B_\epsilon(x)}\phi(y)\,d\lambda(y)$.
 Since we want $\psi$ to be the indicator function of $\ball$, this requirement is 
 clearly satisfied as $\|\phi-\mathbb{A}_\epsilon\phi\|_{\mathscr{L}^1}\lesssim\epsilon$.
 Consequently, for the purposes of Theorem~\ref{main.theorem}, Assumption~(I) is satisfied with 
 $\lambda(n)=\mathcal{O}(n^{-p})$ for any $p<\frac1{\alpha_1}-1$.
Since one can integrate w.r.t.\ $d\nu(\omega)$
 also Assumption~(II) is satisfied with the same $\lambda$. 
 
 Clearly the dimension of $\mu$ is equal to one and Assumption~(V)
 is satisfied with any $d_0<1<d_1$ arbitrarily close to $1$ from the fact that the density functions $h_\omega$ are
 uniformly bounded and bounded away from $0$.
 %Assumption~(IV) is satisfied with any $u_0<1$ again arbitrarily close to $1$. 
 Assumption~(VI) is satisfied with $\xi=\beta=1$. 
 %Assumption~(VIII) follows 
 Also, if we denote by $\psi_{\theta^{-n}\omega}^n$ the (unique) inverse branch of 
 $T_{\theta^{-n}\omega}^n$ 
 which contains the parabolic point $0$, then one has that 
 $|\psi_{\theta^{-n}\omega}^n(I)|=\mathcal{O}(n^{-1/\alpha_1})$ for all $\omega$. Hence 
 $\delta(n)=\mathcal{O}(n^{-\kappa})$ with $\kappa=1/\alpha_1$.
 
 To estimate the distortion we again look at the `worst case' which are
 the parabolic inverse branches $\psi_{\theta^{-n}\omega}^n$ of the map $T_{\theta^{-n}\omega}^n$.
  Put $a_n(\omega)=\psi_{\theta^{-n}\omega}^n(a_0)$,  where $a_0=\frac12$. Then
 $$
 DT_{\theta^{-n}\omega}^n(a_n)=\prod_{j=1}^{n}DT_{\theta^{-j}\omega}(a_j).
 $$
 For the parabolic branch: 
 $DT_{\theta^{-j}\omega}(x)
 =1+(1+\alpha(\theta^{-j}\omega))2^{1+\alpha(\theta^{-j}\omega)}x^{\alpha(\theta^{-j}\omega)})$.
 Also
 $$
 a_{j-1}=T_{\theta^{-j}\omega}(a_j)
 =a_j\!\left(1+2^{1+\alpha(\theta^{-j}\omega)}a_j^{\alpha(\theta^{-j}\omega)}\right)
 =a_j\!\left(DT_{\theta^{-j}\omega}(a_j)\right)^{1/(1+\alpha(\theta^{-j}\omega)}+\mbox{ h.o.t.}
 $$
 and
 $$
 DT_{\theta^{-j}\omega}(a_j)
 =\left(\frac{a_{j-1}}{a_j}\right)^{1+\alpha(\theta^{-j}\omega)}+\mbox{ h.o.t.}
 $$
 which implies the estimate 
 $$
  DT_{\theta^{-n}\omega}^n(a_n)
  \le c_2\prod_{j=0}^{n-1}\left(\frac{a_{j-1}}{a_j}\right)^{1+\alpha(\theta^{-j}\omega)}
  \le c_2\!\left(\prod_{j=0}^{n-1}\frac{a_{j-1}}{a_j}\right)^{1+\alpha_1}
  = c_2\,a_n^{-(1+\alpha_1)}.
 $$
 Denote by $a'_k=a_k(\alpha_0)$ the inverse images of $a_0=\frac12$ in the 
 deterministic case when all maps have the parameter value $\alpha_0$. 
 Then $a'_k\sim c_3k^{-1/\alpha_0}$ for some $c_3>0$. Let us put 
$A_k=(a'_{k-1},a_0]$. In order to estimate the distortion of the maps $T_\omega^n$
on the images of $A_k$ and $A=(a_0,1]$ under the inverse branches
$\mathscr{I}^\omega_n$ of $T_\omega^n$ we look at the `worst case'
when the inverse branch is the parabolic branch $\psi_\omega^n$.
Then 
$$
\Theta(n)
=\frac{DT_{\theta^n\omega}^n(a_n(\omega))}{DT_{\theta^{n+k'}\omega}^n(\tilde{a}_{n+k'}(\omega))}
$$
where $\tilde{a}_{n+k'}=\psi_{\theta^{n+k'}\omega}^n(a'_k)$ and $k'$ is so that 
$a'_k\in\psi_{\theta^{k'}\omega}^{k'}A$.
We estimate the numerator from above by
$$
DT_{\theta^n\omega}^n(a_n(\omega))=
\mathcal{O}(1)\prod_{j=0}^{n-1}\left(\frac{a_j}{a_{j+1}}\right)^{1+\alpha(\theta^{-j}\omega)}
\lesssim \left(\frac{a_0}{a_{n}}\right)^{1+\alpha_1}
\lesssim n^{\frac{1+\alpha_1}{\alpha_0}}
$$
as $a_n\le a'_n=\mathcal{O}(n^{-1/\alpha_0})$.
The denominator is estimated from below as follows:
$$
DT_{\theta^n\omega}^n(\tilde{a}_{n+k'}(\omega))=
\mathcal{O}(1)\prod_{j=k'}^{n+k'-1}\left(\frac{\tilde{a}_j}{\tilde{a}_{j+1}}\right)^{1+\alpha(\theta^{-j}\omega)}
\gtrsim \left(\frac{\tilde{a}_{k'}}{\tilde{a}_{n+k'}}\right)^{1+\alpha_0}.
$$
Since
$|\tilde{a}_{j+k'}-\tilde{a}_{j-1+k'}|
=2^{1+\alpha(\theta^{-j}\omega)}\tilde{a}_{j+k'}^{1+\alpha(\theta^{-j}\omega)}
\le2^{1+\alpha_1}\tilde{a}_{j+k'}^{1+\alpha_1}$
one obtains
$$
\frac{\tilde{a}_{k'}}{\tilde{a}_{n+k'}}
\gtrsim \frac{\tilde{a}_{k'}}{\left(\tilde{a}_{k'}^{-\alpha_1}+n\right)^{-1/\alpha_1}}
$$
where $\tilde{a}_{k'}\sim c_3n^{-\eta/\alpha_0}$ for some $\eta>0$. Hence
$$
\Theta(n)
\lesssim n^{\frac{1+\alpha_1}{\alpha_0}}
\left(n^{\eta\frac{\alpha_1}{\alpha_0}}+n\right)^{-\frac{1+\alpha_0}{\alpha_1}}
n^{\eta\frac{1+\alpha_0}{\alpha_0}}.
$$
If we choose $\eta>0$ so that $\eta\frac{\alpha_1}{\alpha_0}<1$ then these two estimates 
combined yield
$$
\Theta(n)\le c_4{n^{\kappa'}},
$$
where 
$\kappa'=\frac{1+\alpha_1}{\alpha_0}-\frac{1+\alpha_0}{\alpha_1}+\eta\frac{1+\alpha_0}{\alpha_0}$.

Assuming that $0<\alpha_0<\alpha_1<\frac13$, the condition $u_0\kappa-2-\kappa'>1$ of Theorem~\ref{main.theorem} is satisfied since
$$
\kappa-2-\kappa'=\frac1{\alpha_0\alpha_1}
\!\left(\alpha_1^2-\alpha_0^2+\alpha_1(1+\eta-2\alpha_0+\eta\alpha_0)\right)>\frac{(1/3+\eta)}{\alpha_0}>1
$$
is positive for any $\eta>0$ and $u_0$ can be chosen arbitrarily close to $1$.

For the inverse branches $\mathscr{I}^\omega_n$ of $T_\omega^n$ denote by
 $\hat\zeta_\varphi$ the `$n$-cylinder'  that is a pre-image of either $A_k\cup A=(a'_{k-1},1]$
 under the inverse branches $\varphi\in\mathscr{I}^\omega_n$.
We can now nearly use Theorem~\ref{main.theorem}. Note that  if $\mathsf{x}\in (0,1)$
 then for all $n$ large enough  $\mathsf{x}\in A_{n^\theta}\cup A$. If we proceed as in the 
estimate of the term $\mathcal{R}_2$ in Section~\ref{estimate.r2} we obtain
$$
T_\omega^{-j}\ball\cap\ball
\subset\bigcup_{\zeta:\zeta\cap\ball\not=\varnothing}T^{-j}\ball\cap\zeta
=\mathscr{P}_1\cup\mathscr{P}_2
$$
where  the union is over $j$-cylinders $\zeta$ and
$$
\mathscr{P}_1
=\bigcup_{\zeta:\zeta\cap\ball\not=\varnothing}T_\omega^{-j}\ball\cap\hat\zeta,
\hspace{2cm}
\mathscr{P}_2
=\bigcup_{\zeta:\zeta\cap\ball\not=\varnothing}T_\omega^{-j}\ball\cap\zeta\setminus\hat\zeta.
$$
The first set is estimated as before in the main theorem. 
For the second term notice that 
$$
\mathscr{P}_2
=\bigcup_{\varphi\in\mathscr{I}^\omega_j:\varphi(I)\cap\ball\not=\varnothing}
T_\omega^{-j}\ball\cap\varphi(D_{n^\eta})
$$
where $ D_k=(0,a'_{k-1}]$. Hence
$$
\mathscr{P}_2
=\bigcup_{\varphi\in\mathscr{I}^\omega_j:\varphi(I)\cap\ball\not=\varnothing}\varphi(\ball\cap D_{n^\eta})
$$
 is empty for $n$ large enough, i.e.\ so that $a_{n^\eta}<\mathsf{x}$.

%%%%%%%%%%%%%%%%%%%%%%%%%%%%%%%%%%%%%%%%%%%%%%%%%
%%%%%%%%%%%%        THEOREM: PARABOLIC INTERVAL MAPS
%%%%%%%%%%%%%%%%%%%%%%%%%%%%%%%%%%%%%%%%%%%%%%%%%
\begin{thrm} Let $S:\{0,1\}^{\mathbb{Z}}\times I\circlearrowleft$ be the random system described above,
where the maps $T_\omega$ are the parabolic maps $T_{\alpha_0}$ and $T_{\alpha_1}$. 
Assume $0<\alpha_0<\alpha_1<\frac13$.
Denote by $\mu$ the annealed invariant measure and by $\mu^\omega$ the 
fibred measures.
Then for all $t>0$:
\begin{eqnarray*}
\mu^\omega\!\left(y\in[0,1]: 
\tau^\omega_{B_{\rho}(\mathsf{x})}(y)>\frac{t}{\mu(B_{\rho}(\mathsf{x}))}\right)
&\longrightarrow&e^{-t}\\
\mu^\omega_{B_{\rho}(\mathsf{x})}\!\left(y\in [0,1]: 
\tau^\omega_{B_{\rho}(\mathsf{x})}(y)>\frac{t}{\mu(B_{\rho}(\mathsf{x}))}\right)
&\longrightarrow &e^{-t}
\end{eqnarray*}
as $\rho\to0$ for Lebesgue almost every $\mathsf{x}\in [0,1]$ and
 $\nu$-almost every $\omega\in\{0,1\}^{\mathbb{Z}}$.
\end{thrm}

\begin{proof}
We verify the conditions of the theorem. Above we verified Assumptions~(I)--(VI).
Otherwise, as $\xi=1$ and $\alpha_1<\frac13$ one clearly has $\kappa\xi>1$. Also, 
since $u_0$ and $d_1$ can be chosen arbitrarily close to one, $\beta=1$ and choosing $p=\frac1{\alpha'_1}-1$ with $\frac13>\alpha'_1>\alpha_1$
arbitrarily close to $\alpha_1$ we get
$\max\left(\frac{d_1\beta}{\kappa\xi-1},(\frac\beta\xi+d_1)\frac1p\right)\leq\max(\frac{d_1}{2},(1+d_1)\frac{\alpha'_1}{1-\alpha'_1})<\min(1,u_0)$.

\end{proof}

%%%%%%%%%%%%%%%%%%%%%%%%%%%%%%%%%%%%%%%%%%%%%%%%%
%%%%%%%%%%%%%%%%%%%%%%%%%%%%%%%%%%%%%%%%%%%%%%%%%
%%%%%%%%%%%%         SUBSECTION: Random Tower
%%%%%%%%%%%%%%%%%%%%%%%%%%%%%%%%%%%%%%%%%%%%%%%%%
\subsection{Random perturbation of non-uniformly expanding maps with critical set}
Li and Vilarinho constructed in~\cite{LV} a random Gibbs-Markov-Young structure for non-uniformly expanding maps with critical set. Examples of such maps include deformation of a uniformly expanding map by isotopy, and the Viana map in dimension 2. To be more precise, they consider a $C^2$ local diffeomorphism $T$ on a compact Riemannian manifold $M$ with possibly a set $\mathcal{C}\subset M$ consisting of critical points of $T$ and $\partial M$. Then they take a skew product on $\Omega \times M$ where $\Omega = \{1,2,\ldots, k\}^\mathbb{Z}$, so that $T_\omega$ only depends on the first symbol, and $T_\omega = T$ for some $\omega \in \Omega$.

In order to obtain the expanding property for almost every realization $\omega$, they require that the system satisfies the following properties:

\begin{enumerate}
\item there is $\alpha>0$ such that for $\nu \times m$ almost every $(\omega, x)$,
$$
\limsup_{n \to +\infty} \frac1n \sum_{j=0}^{n-1}\log \|DT_{\theta^j\omega}(T^j_\omega(x))^{-1}\|<-\alpha,
$$
where $m$ is the Lebesgue measure on $M$.
\item given any $\gamma>0$ there is $\delta>0$, such that for $\nu \times m$ almost every $(\omega, x)$,
$$
\limsup_{n \to +\infty} -\frac1n \sum_{j=0}^{n-1}\log \dist_\delta(T^j_\omega(x),\mathcal{C})<\gamma;
$$
here $\dist_\delta$ is the truncated distance: $\dist_\delta(x,y) = d(x,y)$ if $d(x,y)<\delta$, and $\dist_\delta(x,y) =1$ otherwise.
\end{enumerate}

Next, define the expansion time function:
$$
\mathcal{E}_\omega(x) = \min \left\{N\ge 1:  \frac1n \sum_{j=0}^{n-1}\log \|DT_{\theta^j\omega}(T^j_\omega(x))^{-1}\|<-\alpha, \mbox{ for all } n \ge N\right\},
$$
and the recurrence time function (to $\mathcal{C}$):
$$
\mathcal{R}_\omega(x) = \min \left\{N\ge 1:  \frac1n \sum_{j=0}^{n-1}\log \dist_\delta(T^j_\omega(x),\mathcal{C})<\gamma, \mbox{ for all } n \ge N\right\}.
$$
In the case when $\mathcal{C}=\emptyset$, condition (2) and $\mathcal{R}_\omega$ can be disregarded. 

Consider the tail set at time $n$ which is defined as 
$$
\Gamma^n_\omega  = \left\{x: \mathcal{E}_\omega(x)>0 \mbox{ or } \mathcal{R}_\omega(x)>n \right\}.
$$
This is the set of point for the realization $\omega$, such that the orbit does not exhibit enough hyperbolicity at time $n$, or get too close to the critical set $\mathcal{C}$. Then they prove the existence of an absolutely continuous probability measure, and quenched decay of correlations:

\begin{thrm}\cite[Theorem A]{LV}\label{LV_decay}
Assume that there is $C,\gamma>0$ and $0<v\le 1$ such that $m(\Gamma^n_\omega)<C e^{-\gamma n^v}$ for $\nu$ almost every $\omega \in \Omega$. Then for some integer $q \ge 1$ we have:
\begin{enumerate}
\item for $\nu$ almost every $\omega$ there is an absolutely continuous probability $\mu_\omega$ such that $(T^q_\omega)_*\mu_\omega = \mu_{\theta^q\omega}$; 
\item there is $C_1, \gamma_1>0$, such that for $\nu$ almost every $\omega$, we have stretched exponential decay of correlation  for  Lipschitz function and $L^\infty$ functions: $$
\left|\int_MG(H\circ T_\omega^k)\,d\mu^\omega
-\mu^\omega(G)\mu^{\theta^k\omega}(H)\right|
\le C_1e^{-\gamma_1n^{v/2}}\|G\|_{Lip}\|H\|_\infty, 
$$
for every $H\in L^\infty(M,\mathbb{R})$ and for every $G\in Lip(M,\mathbb{R})$.
\end{enumerate}
\end{thrm}

This verifies condition (II) for Theorem~\ref{main.theorem}. Condition (I) follows by integrating over $\omega$. Since $\mu^\omega$ are absolutely continuous with respect to the Lebesgue measure for almost every $\omega$, condition (V) and (VI) hold with $d_0 $ and $d_1$ both close to $\dim M$, and $\xi = \beta = 1$. We have the following theorem:

\begin{thrm}
Assume that the assumptions of Theorem~\ref{LV_decay} and  condition (III) and (IV) in Theorem~\ref{main.theorem} hold. Then
\begin{eqnarray*}
\mu^\omega\!\left(y\in M: 
\tau^\omega_{B_{\rho}(\mathsf{x})}(y)>\frac{t}{\mu(B_{\rho}(\mathsf{x}))}\right)
&\longrightarrow&e^{-t}\\
\mu^\omega_{B_{\rho}(\mathsf{x})}\!\left(y\in M: 
\tau^\omega_{B_{\rho}(\mathsf{x})}(y)>\frac{t}{\mu(B_{\rho}(\mathsf{x}))}\right)
&\longrightarrow &e^{-t}
\end{eqnarray*}
as $\rho\to0$ for Lebesgue almost every $\mathsf{x}\in M$ and
 $\nu$-almost every $\omega$.
\end{thrm}

%%%%%%%%%%%%%%%%%%%%%%%%%%%%%%%%%%%%%%%%%%%%%%%%%
%%%%%%%%%%%%%%%%%%%%%%%%%%%%%%%%%%%%%%%%%%%%%%%%%
%%%%%%%%%%%%         SUBSECTION: Random Attractor 
%%%%%%%%%%%%%%%%%%%%%%%%%%%%%%%%%%%%%%%%%%%%%%%%%
\subsection{Random perturbation of partially hyperbolic attractors.}
In~\cite{ABR} the authors consider the small random perturbation of certain partially hyperbolic attractors. To be more precise, they consider a neighborhood $\mathcal F$ of a $C^{1+\alpha}$ diffeomorphism $f\in \mbox{Diff}^{1+\alpha}(M)$ on a compact Riemannian manifold $M$ with dimension at least two. Given any Borel probability measure $\beta$ whose support is a compact subset $B\subset \mathcal F$, they consider the shift space $\Omega = B^{\mathbb Z}$ with invariant probability $\nu = \beta^{\mathbb Z}$. Then the skew product $(\omega, x)\to (\theta\omega, f_\omega(x))$ is a random dynamical system where the map $f_\omega$ only depends on the first symbol of $\omega$.

Then it is assumed that there exists a {\em hyperbolic product structure $\Lambda_\omega$}, consisting of continuous families of stable and unstable leaves (all of which depend on $\omega$), such that 
\begin{itemize}
	\item $\Lambda_\omega$ has positive Lebesgue volume on each unstable leaf;
	\item there is a measurable partition $\{\Lambda_{i,\omega}\}$ of $\Lambda_\omega$ into $u$-subsets and a return time function $R_{i,\omega}$, constant on each $\Lambda_{i,\omega}$, such that the return map is Markov in the sense that 
	$$
	f_\omega^{R_{i,\omega}}(\gamma^u_{_\omega}(x)) \supset \gamma^u_{\theta^{R_{i,\omega}}\omega}(f_\omega^{R_{i,\omega}}(x)),
	$$
	and a similar relation holds for the stable leaves.
	\item $f^n_\omega$ contracts the stable leaves exponentially fast;
	\item on the unstable leaves, the return map $f_\omega^{-R_\omega}$ contracts exponentially fast, with bounded distortion that is independent of $\omega$;
	\item the stable holonomy is absolute continuous, with a $\log$-H\"older density whose H\"older constant is independent of $\omega$.  
\end{itemize} 
Such conditions are similar to Assumption (III) and (IV) in Theorem~\ref{main.theorem}, and can be seen as the random version of the invertible Young's tower in~\cite{LSY98}.  Under the assumption that the tail of the return time is uniformly summable:
$$
\sum_{n\ge 0} \mbox{Leb}_{\gamma^u_{\theta^{-n}\omega}}\left\{R_{\theta^{-n}\omega}>n\right\}\le C
$$
for some constant $C$ that is independent of $\omega$, it is proven that the system admits a family of physical measures $\{\mu^\omega\}$ with $(f_\omega)_*\mu^\omega = \mu^{\theta\omega}$. More importantly, $\mu^\omega$ have absolutely continuous conditional measures along the unstable leaves, whose density w.r.t. the leaf volume is uniformly bounded above and below. Furthermore, if one assumes that the tail of the return time, $\{R_{\theta^{-n}\omega}>n\}$, is exponential/stretch exponential/polynomially small, then the system has quenched and annealed decay of correlations for H\"older continuous functions at corresponding rates.

As a specific example, we take $f\in\mbox{Diff}^{1+\alpha}(M)$ where $M$ is a compact Riemannian manifold, such that $f$ has a topologically mixing uniformly hyperbolic attractor $K\in M$. Assume that $\{f_\omega\}_{\omega\in\Omega}$ is a small random perturbation of $f$. Then it is proven in~\cite[Theorem 1.6]{ABR} that the system has quenched decay of correlations for H\"older functions $G,H$ at exponential speed, with constant $C_{G,H}$ independent of $\omega$. Furthermore, it can be seen from the proof that $H$ is allowed to be $L^\infty$ on local stable leaves.\footnote{For deterministic Young towers, it is a well known fact that the decay of correlations holds for H\"older functions against $L^\infty$ functions that are constant on local stable leaves; the proof of~\cite{ABR} follows the argument of~\cite{LSY98}, and the same conclusion holds.} 

Due to the uniform exponential contracting/expanding on stable and unstable leaves, Assumption (III) and (IV) naturally holds with $\Theta(n)$ being a constant, and $\delta(n)$ decaying exponentially fast. Assumption (V) holds since the density of the conditional measures of $\mu^\omega$ are equivalent to the leaf volume with uniformly bounded density, and  $\gamma^u_\omega$ depends continuously on $\omega$. In the mean time, it is known that Assumption (VI) holds for the physical measure of (the unperturbed map) $f$, see for example~\cite{HW14}. Note that the proof in \cite{HW14} only uses the fact that the conditional measures are absolutely continuous w.r.t. the Lebesgue measure; as a result, a similar proof shows that (VI) holds for the physical measures $\{\mu^\omega\}$.

Therefore we obtain the following theorem:
\begin{thrm}
	Let $\{f_\omega\}$ be a family of same $C^1$ perturbations of $f\in\mbox{Diff}^{\,1+\alpha}(M)$ which has a topologically mixing, uniformly hyperbolic attractor $K\subset M$. Then the random system $S(\omega, x) = (\theta\omega, f_\omega (x))$ and the physical measures $\{\mu^\omega\}$ satisfy $$
	\mu^\omega\!\left(y\in M: 
	\tau^\omega_{B_{\rho}(\mathsf{x})}(y)>\frac{t}{\mu(B_{\rho}(\mathsf{x}))}\right)
	\longrightarrow e^{-t}\qquad \mbox{ as }\rho\to0
	$$
	%for all $t>0$ for $\mu^\omega$-almost every $\mathsf{x}\in M$ and $\nu$-almost every $\omega\in\Omega$.
	and
	%Moreover, if the system also satisfies~(IX), then
	$$
	\mu^\omega_{B_{\rho}(\mathsf{x})}\!\left(y\in M: 
	\tau^\omega_{B_{\rho}(\mathsf{x})}(y)>\frac{t}{\mu(B_{\rho}(\mathsf{x}))}\right)
	\longrightarrow e^{-t}\qquad \mbox{ as }\rho\to0
	$$
	for all $t>0$ for $\mu^\omega$-almost every $\mathsf{x}\in M$ and $\nu$-almost every $\omega\in\Omega$.
\end{thrm}

%\subsubsection{Solenoid with intermittency}

%%%%%%%%%%%%%%%%%%%%%%%%%%%%%%%%%%%%%%%%%%%
%%%%%%%%%%%%%%%%%%%%%%%%%%%%%%%%%%%%%%%%%%%
%%%%%%%%%%%        REFERENCES
%%%%%%%%%%%%%%%%%%%%%%%%%%%%%%%%%%%%%%%%%%%

\end{document}